\documentclass[12pt]{article}
\usepackage{amsmath}
\usepackage{amssymb}
\usepackage{enumerate}
\usepackage{amsbsy}
\usepackage{amsfonts}
\usepackage{amsthm}

\newtheorem{theorem}{Theorem}[section]
\newtheorem{proposition}[theorem]{Proposition}
\newtheorem{lemma}[theorem]{Lemma}

\numberwithin{equation}{section}

\newtheorem{remark}{Remark}

\begin{document}

\title{On the Cauchy problem of fractional Schr\"{o}dinger equation with Hartree type nonlinearity}

\vspace{1cm}
\date{\empty}
\author{Yonggeun Cho\\
{\it \small Department of Mathematics, and Institute of Pure and
Applied Mathematics}\\ {\it \small Chonbuk National University,
Jeonju 561-756, Republic
of Korea}\\ {\it \small e-mail: changocho@jbnu.ac.kr}\\
Gyeongha Hwang\\ {\it \small Department of Mathematical Sciences}\\ {\it
\small Seoul National University, Seoul 151-747, Republic of Korea}\\{\it\small e-mail: ghhwang@snu.ac.kr}\\
Hichem Hajaiej\\{\it\small Department of Mathematics}\\{\it\small
King
Saud University, P.O. Box 2455, 11451 Riyadh, Saudi Arabia}\\
{\it\small e-mail:
hhajaiej@ksu.edu.sa}\\
Tohru Ozawa\\{\it\small Department of Applied Physics}\\{\it\small
Waseda University, Tokyo 169-8555, Japan}\\{\it\small e-mail:
txozawa@waseda.jp}}

\maketitle

\vspace{1cm}

\begin{abstract}
We study the Cauchy problem for the fractional Schr\"{o}dinger
equation
$$
i\partial_tu = (m^2-\Delta)^\frac\alpha2 u +
F(u)\;\;\mbox{in}\;\;\mathbb{R}^{1+n},
$$
where $ n \ge 1$, $m \ge 0$, $1 < \alpha < 2$, and $F$ stands for the nonlinearity
of Hartree type:
$$F(u) = \lambda \left(\frac{\psi(\cdot)}{|\cdot|^\gamma} *
|u|^2\right)u$$ with $\lambda = \pm1, 0 <\gamma < n$, and $0 \le
\psi \in L^\infty(\mathbb R^n)$. We prove the existence and
uniqueness of local and global solutions for certain $\alpha$,
$\gamma$, $\lambda$, $\psi$. We also remark on finite time blowup of
solutions when $\lambda = -1$.\\



\end{abstract}

\section{Introduction}

 In this paper we consider the following Cauchy problem:
\begin{align}\label{main eqn}\left\{\begin{array}{l}
i\partial_tu = D_m^\alpha u + F(u),\;\;
\mbox{in}\;\;\mathbb{R}^{1+n}
\times \mathbb{R},\; n \ge 1\\
u(x,0) = \varphi(x)\;\; \mbox{in}\;\;\mathbb{R}^n,
\end{array}\right.
\end{align}
where $D_m = (m^2-\Delta)^\frac12$, $1 < \alpha < 2$, and $F(u)$ is
nonlinear functional of Hartree type such that $F(u) = \lambda
\left(\frac{\psi(\cdot)}{|\cdot|^\gamma}
* |u|^2\right)u \equiv \lambda K_\gamma(|u|^2)u$, where $*$ denotes the convolution in
$\mathbb{R}^n$, $\lambda = \pm 1$, $\mu \ge 0$, $0 < \gamma < n$ and
$0 \le \psi \in L^\infty(\mathbb R^n)$.

When $m = 0$, the equation \eqref{main eqn} is called fractional
Schr\"{o}dinger equation which was used to describe particles in
L\'{e}vy stochastic process, and when $m > 0$, generalized semirelativistic
equation. See \cite{la1, la2, la3, lenz} and the
references therein.

If $m = 0$, then similarly to the Schr\"{o}dinger case $(\alpha =
2)$ the equation \eqref{main eqn} has scaling invariance property.
In fact the function $u_a(t, x) = a^\frac{n-\gamma+\alpha}2
u(a^\alpha t, a x)\;(a > 0)$ is also a solution of \eqref{main eqn}.
The associated invariant space is $\dot H^\frac{\gamma-\alpha}2$.
So, we call the equation $\dot H^\frac{\gamma-\alpha}2$-subcritical
if we pursue the solution $u \in H^s$ for $s >
\frac{\gamma-\alpha}2$, $\dot H^\frac{\gamma-\alpha}2$-critical for
$s = \frac{\gamma-\alpha}2$ and $\dot
H^\frac{\gamma-\alpha}2$-supercritical for $s <
\frac{\gamma-\alpha}2$.

The purpose of this paper is to establish the local and global
existence theory to the equation \eqref{main eqn} and also finite time
blowup. In this paper we study the Cauchy problem \eqref{main eqn}
in the form of the integral equation:
\begin{align}\label{int eqn}
u(t) = U(t)\varphi - i\int_0^t U(t-t')F(u)(t')dt',
\end{align}
where $$U(t)\varphi(x) =
(e^{-it(m^2-\Delta)^\frac\alpha2}\varphi)(x) = \frac1{(2\pi)^n}
\int_{\mathbb{R}^n} e^{i(x\cdot \xi - t
(m^2+|\xi|^2)^\frac\alpha2)}\widehat{\varphi}(\xi)\,d\xi.$$ Here
$\widehat{\varphi}$ denotes the Fourier transform of $\varphi$ such
that $\widehat{\varphi}(\xi) = \int_{\mathbb{R}^n} e^{-ix\cdot \xi}
\varphi(x)\,dx$.

One of the key tools for the global theory is the conservation law.
If the solution $u$ of \eqref{main eqn} has sufficient decay at
infinity and smoothness, it satisfies two conservation laws:
\begin{align}\begin{aligned}\label{consv}
&\qquad\quad\|u(t)\|_{L^2} = \|\varphi\|_{L^2},\\
&E(u) \equiv K(u) + V(u) = E(\varphi),
\end{aligned}\end{align}
where $K(u) = \frac12 \langle (m^2- \Delta)^\frac\alpha2\,u,
u\rangle$, $V(u) = \frac14 \langle F(u), u \rangle$ and
$\langle\;,\rangle$ is the complex inner product in $L^2$. The energy
space is $H^\frac\alpha2$. So, the equation \eqref{main eqn} is
referred to be energy critical if $\gamma = 2\alpha$, subcritical if
$\gamma < 2\alpha$ and supercritical if $\gamma
> 2\alpha$, respectively. Similarly we use the terminology mass critical, subcritical, supercritical for the case $\gamma = \alpha$,
$\gamma < \alpha$, $\gamma > \alpha$, respectively. For the proof of
\eqref{consv} a regularizing method is simply applicable as in
\cite{lenz} in the case of $0 < \gamma \le \alpha$. For local
solutions constructed by a contraction argument based on the
Strichartz estimate stated below, the case of $\alpha < \gamma \le
2\alpha$ is treated by exactly the same method as in \cite{oz}
without using approximate or regularizing approach. The second tool
is the Strichartz estimates. In Section 2 we recall three versions
which will be used in the argument of the paper.

In Section 3, without resort to Strichartz estimates local and
global existence results are treated for $m \ge 0$ through the contraction
argument and the conservation laws above. This result is an
extension of the work of Lenzmann \cite{lenz} and \cite{chooz} to
fractional NLS. In particular, we show the global existence in the
focusing mass critical case, that is, $\gamma = \alpha, \lambda = -1$, for
the initial norm with $\|\varphi\|_{L^2} <
\|Q\|_{L^2}/\|\psi\|_{L^\infty}^\frac12$, where $Q$ is the solution of $
(-\Delta )^\frac\alpha2 Q - (|x|^{-\gamma}* |Q|^2)Q = -Q.$ We also
show the solution norm can be estimated uniformly in terms of $m$ in
finite time, which enables us to consider two types of limiting
problems ($m \to 0$ and $m \to \infty$). See Remark \ref{limiting}
below.

In Section 4, we consider the local and global existence via
standard Strichartz estimates \eqref{homo str} and \eqref{inhomo
str} below. The advantage of Strichartz estimate is to give a chance
for existence results of lower regularity than ones without using Strichartz estimates.
However, owing to the regularity loss of Strichartz estimates, it is hard to handle  the
critical problem. On the other hand, such
estimates enable us to get a small data global existence results and
scattering for the case $2\alpha < \gamma < n$.

In Section 5, we treat the critical problem. To handle the critical
regularity one needs Strichartz estimate without regularity loss as
Schr\"{o}dinger case. Recently, such estimates have been developed
independently in \cite{guwa} and \cite{cholee}, when radial symmetry
or angular regularity is assumed. See \eqref{rad str1} and
\eqref{rad str2} below. Using these, we show the global existence of
radial solutions in $H^\frac{\gamma-\alpha}2$ for suitable $\gamma$
and $\alpha$. In \cite{guwa}, the authors considered the equation
with $m = 0$ and power type nonlinearity.

Section 6 is devoted to the global existence of small data in
critical solution space below $L^2$, that is $\dot
H^\frac{\gamma-\alpha}2, \gamma < \alpha$ without radial symmetry.
For this we use weighted Strichartz estimates \eqref{lomegainfty}
and \eqref{lomega2} in the same way as in \cite{chhwoz}. When $m >
0$, we could not control the homogeneous $\dot H^s$ norm by the
weighted Strichartz estimates. Thus we only consider the case $m =
0$. It would be so interesting to show the global existence when $m
> 0$. For the simplicity of presentation we try 3-d case in Section 6. We leave
the general case to the readers.

In the last section, we study a finite time blowup for the focusing case. For this we consider a massive mass critical Hartree nonlinearity given by the mass $m > 0$ and the potential $-\psi(x)/|x|^\alpha$ where $\psi' \leq 0$ and $|\psi'| \lesssim \frac 1{\rho}$, and a initial data with $E(\varphi) < 0$. Then by adapting the Virial argument of \cite{frohlenz2} and \cite{chkl} we show the nonnegative
quantity $\langle u, x \cdot D_m^{2-\alpha} x u \rangle$ is estimated as
follows: for any $m \ge 0$ and $t \in [0, T^*)$
\begin{align}\label{mom}
\langle u, x \cdot D_m^{2-\alpha} x u \rangle \le 2\alpha^2E(\varphi)t^2 + 2\alpha(\langle \varphi, A \varphi \rangle+ C\|\varphi\|_{L^2}^4)t + \langle \varphi, M \varphi\rangle.
\end{align}
Since $E(\varphi) < 0$, the maximal existence time $T^*$ of solution should be finite. In \cite{chkl}, the authors considered massless case and they obtained finite time blowup for mass critical equations. We extended their results to massive case and show that the constant $C$ in \eqref{mom} does not depend on $m > 0$. For the proof of \eqref{mom} we show $L^2$ operator norm of the commutator $[D_m^\alpha, |x|^2 K_\alpha(|u|^2)]$ is bounded by
$\|\varphi\|_{L^2}^4$ for which we need to assume that radial symmetry of solution. We also establish some propagation estimates of moment at the end of Section 7. 




Now we close this section by introducing some notations.
The mixed norm $\|F\|_{L^qX}$ means $(\int_{\mathbb{R}}\|F(t, \cdot)\|_X^q\,dt)^\frac1q$. We will use
the notations $|\nabla| = \sqrt{-\Delta}$, $\dot H_r^s =
|\nabla|^{-s}L^r\;\;(\dot H_2^s = \dot H^s)$ and $H_r^s = (1 -
\Delta)^{-s/2} L^r\;\; (H^s = H_2^s)$. Hereafter, we denote the
space $L_T^q (B)$ by $L^q(0,T; B)$ and its norm by
$\|\cdot\|_{L_T^qB}$ for some Banach space $B$, and also $L^q(B)$
with norm $\|\cdot\|_{L^q B}$ by $L^q(0, \infty; B)$, $1 \le q \le
\infty$. If not specified, throughout this paper, the notation $A \lesssim B$
and $A \gtrsim B$ denote $A \le CB$ and $A \ge C^{-1}B$,
respectively. Different positive constants possibly depending on $n, \alpha$
and $\gamma$ might be denoted by the same letter $C$. $A \sim B$
means that both $A \lesssim B$ and $A \gtrsim B$ hold.

\section{Strichartz estimates}
In this paper we will
treat three versions of Strichartz estimates. We first consider the
standard Strichartz estimate for the unitary group $U(t)$ (see
\cite{cox}):
\begin{align}
&\|U(t)\varphi\|_{L_T^{q_1}L^{r_1}} \le C c_\alpha^{\frac12-\frac1{r_1}}\|D_m^{\frac{n(2-\alpha)}2\left(\frac12- \frac1{r_1}\right)}\varphi\|_{L^2},\label{homo str}\\
&\|\int_0^t U(t-t')F(t')\,dt'\|_{L_T^{q_1}L^{r_1}} \le
Cc_\alpha^{1-\frac1{r_1} - \frac1{r_2}}\|\mathcal
D_m^{\frac{n(2-\alpha)}2\left(1-\frac1{r_1}-\frac1{r_2}\right)}
F\|_{L_T^{q_2'}L^{r_2'}},\label{inhomo str}
\end{align}
where $c_\alpha = (\alpha-1)^{-1}$ and the constant $C$ does not depend
on $m$. These estimates hold for $n \ge 1$ and the pairs $(q_i,
r_i), i = 1, 2$ satisfying that $2 \le q_i, r_i \le \infty$, $\frac
2{q_i} + \frac{n}{r_i} = \frac n2$ and $(q_i, r_i) \neq (2,
\infty)$.  The constant $c_\alpha$ shows the sharpness of the estimates near $\alpha = 1$. We will use the estimates \eqref{homo str} and
\eqref{inhomo str} for the existence of $H^s$ solutions for some $s < \frac \gamma2$ in Section 5.

Next we will use the recently developed radial Strichartz estimates
\cite{cholee, guwa} as follows: for radial functions $\varphi$ and
$F$
\begin{align}
&\qquad \quad \|U(t)\varphi\|_{L_T^{q_1}L^{r_1}} \le Cc_\alpha^{\frac12-\frac1{r_1}}\|D_m^\theta \varphi\|_{L^2},\label{rad str1}\\
&\|\int_0^t U(t-t')F(t')\,dt'\|_{L_T^{q_1}L^{r_1}} \le C
c_\alpha^{1-\frac1{r_1} -
\frac1{r_2}}\|F\|_{L_T^{q_2'}L^{r_2'}},\label{rad str2}
\end{align}
where $C$ does not depend on $m$. Here $\theta \in \mathbb R$ and $n
\ge 2$. The pairs $(q_i, r_i), i = 1, 2$, satisfy the range
conditions $2 \le q_i, r_i \le \infty, q_2 \neq 2$,
$$\frac n2\left(\frac12-\frac1{r_i}\right) \le \frac1{q_i} \le \frac{2n-1}2\left(\frac12-\frac1{r_i}\right),$$
$$(n, q_i, r_i) \neq (2, 2, \infty), \quad (q_i, r_i) \neq (2,
\frac{2(2n-1)}{2n-3}),$$ and the gap condition
$$
\frac\alpha{q_1} + \frac n{r_1} = \frac n2 - \theta,\qquad
\frac\alpha{q_2} + \frac n{r_2} = \frac n2 + \theta.
$$
These will be used for global well-posedness of radial solution with critical regularity in Section 6.

Finally to treat the well-posedness in the case of below $L^2$ we will use the weighted Strichartz estimates:\\
(1) Let $0 < a < \frac{n-1}2$ and $\beta_1 \leq \frac{n-1}2 - a$.
Then we have
\begin{align}\label{lomegainfty}
\||x|^a |\nabla|^{a - \frac n2}d_\omega^{\beta_1}
U(t)\varphi\|_{L_t^\infty L_r^\infty L_\omega^2} \le
C\|\varphi\|_{L_x^2}.
\end{align}
\noindent (2) Let $-\frac n2 < b < -\frac 12$ and $\beta_2 \le
-\frac12 -b $. Then we have
\begin{align}\label{lomega2}
\||x|^b
|\nabla|^{1+b}D_m^{-\frac{2-\alpha}2}d_\omega^{\beta_2}U(t)\varphi\|_{L_{t,x}^2}
\le C \|\varphi\|_{L_x^2}.
\end{align}

Here $d_\omega = \sqrt{1 - \Delta_\omega}$, $\Delta_\omega$ is
the Laplace-Beltrami operator on the unit sphere $S^{n-1}$ and $C$ does not depend on $m$. We have used the notation $\|f\|_{L_r^{r_1}L_\omega^{r_2}} = (\int_0^\infty (\int_{S^{n-1} } |f(\rho\omega)|^{r_2}\,d\omega)^\frac{r_1}{r_2}\,\rho^{n-1}\,d\rho)^\frac1{r_1}$.
For the part(1) see \cite{chooz} and \cite{fawa}. For \eqref{lomega2} we refer to \cite{fawa} and also to
\cite{chozsash, choozlee} for earlier and more general versions,
respectively.

\section{Existence I}
In this section, we study the local and global existence without
resort to Strichartz estimates.

Let us first introduce the following local existence result.
\begin{proposition}\label{local}
Let $m \ge 0$, $0 < \gamma < n$ and $n \ge 1$. Suppose $\varphi \in
H^s(\mathbb{R}^n)$ with $s \ge \frac \gamma 2$. Then there exists a
positive time $T$ such that \eqref{int eqn} has a unique solution $u
\in C([0,T]; H^s)$ with $\|u\|_{L_T^\infty H^s} \le
C\|\varphi\|_{H^s}$, where $C$ does not depend on $m \ge 0$.
\end{proposition}
\begin{proof} Let $(X(T,\rho), d)$ be a complete metric space with metric
$d$ defined by
$$
X(T, \rho) = \{u \in L_T^\infty (H^s(\mathbb{R}^n)):\;
\|u\|_{L_T^\infty H^s} \le \rho\}, \;\; d_X(u, v) = \|u -
v\|_{L_T^\infty L^2}.
$$

Now we define a mapping $\mathcal N: u \mapsto \mathcal N(u)$ on
$X(T, \rho)$ by
\begin{align}\label{nonlinear func}
\mathcal N(u)(t) = U(t)\varphi - i\int_0^t U(t-t')F(u)(t')\,dt'.
\end{align}
Our strategy is to use the standard contraction mapping argument. To
do so, let us introduce a generalized Leibniz rule (see Lemma A1
$\sim$ Lemma A4 in Appendix of \cite{ka}).
\begin{lemma}\label{lei}
For any $s \ge 0$ we have
$$
\||\nabla|^s(uv)\|_{L^r} \lesssim \||\nabla|^s
u\|_{L^{r_1}}\|v\|_{L^{q_2}} + \|u\|_{L^{q_1}}\||\nabla|^s v\|_{L^{r_2}},
$$
where $\frac1{r} = \frac1{r_1} + \frac1{q_2} =
\frac1{q_1}+\frac1{r_2},\quad r_i \in (1, \infty),\; q_i \in
(1,\infty],\quad i = 1, 2.
$
\end{lemma}
\medskip

Then for $u \in X(T, \rho)$ and $s \ge \frac\gamma2$ we have
\begin{align}\begin{aligned}\label{cont1}
\|\mathcal N(u)\|_{L_T^\infty H^s}&\le \|\varphi\|_{H^s} + T\|F(u)\|_{L_T^\infty H^s}\\
&\lesssim \|\varphi\|_{H^s} +
T\left(\|K_{\gamma}(|u|^2)\|_{L_T^\infty
L^\infty}\|u\|_{L_T^\infty H^s}\right.\\
&\qquad\qquad\qquad\qquad\qquad\left. +
\|K_{\gamma}(|u|^2)\|_{L_T^\infty H_{\frac{2n}{\gamma}}^s}\|u\|_{L_T^\infty L^\frac{2n}{n - \gamma}}\right)\\
&\lesssim \|\varphi\|_{H^s} + T\left( \|u\|_{L_T^\infty H^\frac
\gamma2}^2\|u\|_{L_T^\infty H^s} + \|u\|_{L_T^\infty
L^\frac{2n}{n-\gamma}}^2\|u\|_{L_T^\infty H^s}\right)\\
&\lesssim \|\varphi\|_{H^s} + T \|u\|_{L_T^\infty
H^\frac\gamma2}^2\|u\|_{L_T^\infty H^s} \lesssim \|\varphi\|_{H^s} +
T\rho^3.
\end{aligned}\end{align}
Here we have used the trivial inequality $$K_\gamma(v) =
\int_{\mathbb{R}^n} \frac{\psi(x-y)}{|x-y|^\gamma}v(y)\,dy \le
\|\psi\|_{L^\infty}\int_{\mathbb R^n} |x-y|^{-\gamma}v(y)\,dy$$ for
$v \ge 0$,
the Hardy-Littlewood-Sobolev inequality,
Lemma \ref{lei}, the Hardy type inequality
\begin{align}\label{supx frac}
\sup_{x \in \mathbb{R}^n}\left|\int_{\mathbb{R}^n}
\frac{|u(x-y)|^2}{|y|^\gamma}\,dy\right| \lesssim
\|u\|_{\dot{H}^\frac \gamma2}^2,
\end{align}
and we used the Sobolev embedding $H^\frac \gamma2
\hookrightarrow L^\frac{2n}{n - \gamma}$.

If we choose $\rho$ and $T$ such as $\|\varphi\|_{H^s} \le \rho/2$
and $CT\rho^3 \le \rho/2$, then $\mathcal N$ maps $X(T, \rho)$ to
itself.

Now we show that $\mathcal N$ is a Lipschitz map for sufficiently
small $T$. Let $u, v \in X(T, \rho)$. Then we have
\begin{align*}
d_X(\mathcal N(u)&, \mathcal N(v))\\
&\lesssim T\left\|K_{\gamma}(|u|^2)u -
K_\gamma(|v|^2)v \right\|_{L_T^\infty L^2}\\
&\lesssim T\left(\left\|K_\gamma(|u|^2)(u - v)\right\|_{L_T^\infty
L^2} + \left\|K_\gamma(|u|^2 -
|v|^2)v\right\|_{L_T^\infty L^2}\right)\\
&\lesssim T\left(\|u\|_{L_T^\infty H^\frac \gamma2}^2d(u, v) +
\|K_\gamma(|u|^2 - |v|^2)\|_{L_T^\infty
L^\frac{2n}{\gamma}}\|v\|_{L_T^\infty
L^\frac{2n}{n-\gamma}}\right)\\
&\lesssim T(\rho^2d(u, v) + \rho\||u|^2 -
|v|^2\|_{L_T^\infty L^\frac{2n}{2n-\gamma}})\\
&\lesssim T(\rho^2  + \rho(\|u\|_{L_T^\infty L^\frac{2n}{n-\gamma}}
+ \|v\|_{L_T^\infty
L^\frac{2n}{n-\gamma}}))d_X(u, v)\\
&\lesssim T\rho^2 d_X(u, v).
\end{align*}
The above estimate implies that the mapping $\mathcal N$ is a
contraction, if $T$ is sufficiently small.

The uniqueness and time continuity follows easily from the equation
\eqref{int eqn} and a similar contraction argument. This completes
the proof of Proposition \ref{local}. \end{proof}

\bigskip

From the conservation laws \eqref{consv}, we get the following
global well-posedness.
\begin{theorem}\label{global1}
Let $m \ge 0$, $0 < \gamma \le \alpha$ for $n\ge 2$, $0 < \gamma <
1$ for $n = 1$, and $s \ge \frac \gamma2$. Let $T^*$ be the maximal
existence time of the solution $u$ as in Proposition $\ref{local}$. Then
if $\lambda = +1$, or if $\lambda = -1$ and $\|\varphi\|_{L^2}$ is
sufficiently small, then $T^* = \infty$. Moreover $\|u(t)\|_{H^s}
\le C\|\varphi\|_{H^s}e^{C(|E(\varphi)| + \|\varphi\|_{L^2}^2)t}$,
where $C$ does not depend on $m \ge 0$.
\end{theorem}

\begin{proof}
From the estimate \eqref{supx frac} and $L^2$ conservation, we have
\begin{align}\label{potential est}
|V(u)| \lesssim \|u\|_{\dot{H}^\frac\gamma2}^2\|u\|_{L^2}^2.
%
\end{align}
Thus if $\lambda = +1$ or if $\lambda = -1$ and $\|\varphi\|_{L^2}$
is sufficiently small, then since $\gamma \le \alpha$
\begin{align}\label{uniform h12}
\|u(t)\|_{\dot{H}^\frac\gamma2}^2 \le C(|E(u)| +
\|\varphi\|_{L^2}^2) = C(|E(\varphi)| + \|\varphi\|_{L^2}^2).
\end{align}
From \eqref{uniform h12} and a similar estimate to \eqref{cont1}, we
have
\begin{align}\begin{aligned}\label{gronwal}
\|u(t)\|_{H^s} &\lesssim \|\varphi\|_{H^s} + \int_0^t
\|u\|_{H^\frac\gamma2}^2\|u\|_{H^s}\,dt'\\
&\lesssim \|\varphi\|_{H^s} + (|E(\varphi)| +
\|\varphi\|_{L^2}^2)\int_0^t\|u\|_{H^s}\,dt'.
\end{aligned}\end{align}
Gronwall's inequality shows that
$$
\|u(t)\|_{H^s} \le C\|\varphi\|_{H^s}\exp(C(|E(\varphi)| +
\|\varphi\|_{L^2}^2)t).
$$
This completes the proof of Theorem \ref{global1}.
\end{proof}

If $\psi = 1$, $m \ge 0$, $\gamma = \alpha$ and $\lambda = -1$, then
\eqref{main eqn} is $L^2$-critical focusing FNLS and more precise
statement is possible for global existence. In fact, \eqref{main
eqn} has a ground state $Q$ in $H^\frac\alpha2$ (see Theorem 1.8 of
\cite{hmow}), which satisfies
$$
(-\Delta )^\frac\alpha2 Q - (|x|^{-\gamma}* |Q|^2)Q = -Q
$$
and is a decreasing minimizer of the problem
\begin{align}\label{mini}
2\|Q\|_{L^2}^2 = \inf_{u \in H^\frac\alpha2 \setminus
\{0\}}\frac{\|u\|_{\dot H^\frac\alpha2}^2\|u\|_{L^2}^2}{|V_1(u)|},
\end{align}
where $V_1(u) = -\frac14\int\!\!\int |x-y|^{-\gamma}|u(x)|^2|u(y)|^2\,dxdy$.
Then we have the following.

\begin{theorem}\label{global1-1}
Let $m \ge 0$, $\gamma = \alpha$, $n\ge 2$ and $s \ge \frac
\gamma2$. Suppose $T^*$ be the maximal existence time of the
solution $u$ as in Theorem $\ref{local}$. Then if $\lambda = -1$ and
$\|\varphi\|_{L^2} < \|Q\|_{L^2}/\|\psi\|_{L^\infty}^\frac12$, then $T^* =
\infty$.
\end{theorem}
\begin{proof}
From \eqref{mini} we estimate $E(u)$ as follows.
\begin{align*}
E(\varphi) = E(u) &= \frac12 \|D_m^\frac\alpha2 u\|_{L^2}^2 - |V(u)|\\
&\ge \frac12 \||\nabla|^\frac\alpha2 u\|_{L^2}^2 -
\frac{\|\psi\|_{L^\infty}}{2\|Q\|_{L^2}^2}\||\nabla|^\frac\alpha2
u\|_{L^2}^2\|u\|_{L^2}^2\\
&=
\frac12\left(1-\frac{\|\psi\|_{L^\infty}\|u\|_{L^2}^2}{\|Q\|_{L^2}^2}\right)\||\nabla
|^\frac\alpha2 u\|_{L^2}^2\\
&=
\frac12\left(1-\frac{\|\psi\|_{L^\infty}\|\varphi\|_{L^2}^2}{\|Q\|_{L^2}^2}\right)\||\nabla
|^\frac\alpha2 u\|_{L^2}^2.
\end{align*}
Thus $\|u\|_{H^\frac\alpha2}^2 \lesssim E(\varphi) +
\|\varphi\|_{L^2}^2$, provided $\|\varphi\|_{L^2} <
\|Q\|_{L^2}/\|\psi\|_{L^\infty}^\frac12$. In the same way as in the proof of
Theorem \ref{global1} we conclude the global existence. That is $T^*
= \infty$.
\end{proof}

\begin{remark}\label{limiting}
Let $T^*_{fnls} = \inf_{m \ge 0}T^*$, where $T^*$ is the maximal existence time of local solution $u_m$ in Proposition \ref{local}. Then from the uniform estimate of solution in $H^s$ norm with respect to $m \ge 0$ it follows that $T^*_{fnls} > 0$. This gives two types of limit problems as Propositions 2.4 and
2.5 of \cite{chooz}. For each $m > 0$ let $u_m \in C([0,T_{fnls}^*); H^s)$ be the solution
of \eqref{main eqn} for $s \ge \frac{\gamma}{2}$ and $u_0$ be the
$H^s$ solution to the Cauchy problem:
$$
i\partial_t u_0 = (-\Delta)^\frac\alpha2 u_0 + F(u_0),\;\;u_0(0) =
\varphi.
$$
Then it immediately follows that for any $T < T_{fnls}^*$ $u_m \to u_0$ in
$C([0,T]; H^s)$ as $m \to 0$.

On the other hand, let $v_m = e^{itm} u_m$, the phase modulation of
the solution $u_m$ to \eqref{main eqn}. Then $v_m$ is the solution
in $C([0,T_{fnls}^*); H^s)$ to the problem
$$
i \partial_t v_m = ((m^2 - \Delta)^\frac\alpha2 - m^\alpha)v_m +
F(v_m), \;\;v_m(0) = \varphi,
$$
and if $w_m$ be the solution in $C([0,T_{nls}^*); H^s)$ to
$$
i\partial_t w_m = -\frac{\alpha }{2m^{2-\alpha}}\Delta w_m +
F(w_m),\;\;w_m(0) = \varphi.
$$
Here $T_{nls}^*$ is the infimum of maximal existence time of $w_m$ with respect to $m$ and the uniform estimate of $w_m$ similar to $u_m$ implies that $T_{nls}^* > 0$. Then by the same argument as in the proof of Proposition 2.5 of
\cite{chooz} one can also show that $\|v_m - w_m\|_{L^\infty(0,T;
H^s)} \to 0$ as $m\to \infty$ for any $T < \min(T_{fnls}^*,
T_{nls}^*)$.

\end{remark}

\bigskip

\section{Existence II: via Strichartz estimates}
In this section, we show the existence results with slightly lower regularity than the previous by using Strichartz estimates \eqref{homo str} and \eqref{inhomo str}. The following is on the local existence.
\begin{proposition}\label{local2}
Let $n \ge 1$, $m \ge 0$ and $s > \frac\gamma2 - \min(\gamma, 2)\frac\alpha4$ for $1 < \alpha < 2$ and $0 < \gamma < n$. If $\varphi
\in H^s$ then there exists a positive time $T$ such that \eqref{int eqn}
has a unique solution $u \in C([0,T]; H^s) \cap L^q_T
(H_r^{s-\sigma})$, where $q = \frac{4}{\delta}$, $r =
\frac{2n}{n-\delta}$ and $\sigma = \frac{\delta(2-\alpha)}4$ for some $\delta$ with $0 < \delta < \min(\gamma, 2)$ and $s > \frac\gamma2-\frac{\delta\alpha}4$.
\end{proposition}
\begin{proof}
Given $n, \alpha, \gamma$ and $s$, choose a number $\delta$ with $0 < \alpha <
\min(\gamma, 2)$ and $s > \frac{\gamma}{2} -
\frac{\delta\alpha}{4}$. Then for some positive number $T$ to be
chosen later, let us define a complete metric space $(Y(T, \rho),
d_Y)$ with metric $d_Y$ by
\begin{align*}
&Y(T, \rho) = \left\{v \in L_T^\infty (H^s) \cap
L_T^q(H_r^{s-\sigma}) :
\|v\|_{L_T^\infty H^s} + \|v\|_{L_T^q H_r^{s-\sigma}} \le \rho\right\},\\
&d_Y(u, v) = \|u - v\|_{L_T^\infty H^s \cap L_T^q H_r^{s-\sigma}},
\end{align*}
where $q, r, \sigma$ are the same indices as in Proposition
\ref{local2}.

We will show that the mapping $\mathcal N$
defined by \eqref{nonlinear func} is a contraction on $Y(T,
\rho)$, provided $T$ is sufficiently small. For this purpose we introduce a useful lemma.
\begin{lemma}[Lemma 3.2 of \cite{chooz}]\label{frac bound}
Let $0 < \gamma < n$. Then for any $0 < \varepsilon < n - \gamma$ we
have
$$
\left\|K_\gamma(|u|^2)\right\|_{L^\infty} \lesssim
\|u\|_{L^\frac{2n}{n-\gamma-\varepsilon}}\|u\|_{L^{\frac{2n}{n-\gamma+\varepsilon}}}.
$$
\end{lemma}

If we use the Strichartz estimates \eqref{homo str} and \eqref{inhomo str} with
the pair $$(q_1, r_1, q_2, r_2) = \left(q = \frac{4}{\delta},
r = \frac{2n}{n-\delta}, \infty, 2\right)$$ together with
Plancherel theorem, Lemma \ref{frac bound}, and generalized Leibniz
rules (Lemma \ref{lei}), then since $\sigma =  \frac{\delta(2-\alpha)}4$ we have
\begin{align}\begin{aligned}\label{map1}
\|\mathcal N(u)&\|_{L_T^\infty H^s \cap L_T^q H_{r}^{s-\sigma}}\\
&\lesssim \|\varphi\|_{H^s} +
\|D_m^{\frac{2-\alpha}2 n(\frac12-\frac1{r})}F(u)\|_{L_T^1H^{s-\sigma}}\\
&\lesssim \|\varphi\|_{H^s} +
\|K_\gamma(|u|^2)\|_{L_T^1L^\infty}\|u\|_{L_T^\infty H^s}\\
&\qquad\qquad\qquad + \int_0^T
\|K_\gamma(|u|^2)\|_{H_{\frac{2n}{\gamma+\varepsilon}}^s}\|u\|_{L^\frac{2n}{n-(\gamma+\varepsilon)}}\,dt\\
&\lesssim \|\varphi\|_{H^s} +
\|u\|_{L_T^2L^{\frac{2n}{n-(\gamma+\varepsilon)}}}\|u\|_{L_T^2L^\frac{2n}{n-(\gamma-\varepsilon)}}\|u\|_{L_T^\infty H^s}\\
&\qquad\qquad\qquad + \int_0^T \||u|^2\|_{H_{\frac{2n}{2n -
(\gamma-\varepsilon)}}^s}\|u\|_{L^\frac{2n}{n-(\gamma+\varepsilon)}}\,dt\\
&\lesssim \|\varphi\|_{H^s} +
\|u\|_{L_T^2L^{\frac{2n}{n-(\gamma+\varepsilon)}}}\|u\|_{L_T^2L^\frac{2n}{n-(\gamma-\varepsilon)}}\|u\|_{L_T^\infty
H^s}
\end{aligned}\end{align}
for sufficiently small $\varepsilon$. Here the involved constant is uniform on $m$ if $ 0 \le m \le m_0$.

Using H\"{o}lder's inequality for time integral, we have
\begin{align}\begin{aligned}\label{map2}
\|\mathcal N(u)&\|_{L_T^\infty H^s \cap L_T^q H_r^{s-\sigma}}\\
&\lesssim \|\varphi\|_{H^s} + T^{1 - \frac 2q}\|u\|_{L_T^q
L^{\frac{2n}{n-(\gamma+\varepsilon)}}}\|u\|_{L_T^qL^\frac{2n}{n-(\gamma-\varepsilon)}}\|u\|_{L_T^\infty
H^s}.
\end{aligned}\end{align}
Now if we choose $\varepsilon > 0$ so small that $\varepsilon <
\min\left(\gamma-\delta, 2(s - \sigma)
-\gamma\right)$, then since $$\frac{2n}{n-\delta} \le
\frac{2n}{n-(\gamma-\varepsilon)} <
\frac{2n}{n-(\gamma+\varepsilon)} \le
\frac{2n}{n-\delta-2(s-\sigma)},$$ we have from \eqref{map2} and
Sobolev embedding $H_r^{s-\sigma} \hookrightarrow L^r \cap
L^\frac{2n}{n-\delta-2(s-\sigma)}$ that
\begin{align*}
\|\mathcal N(u)\|_{L_T^\infty H^s \cap L_T^q H_r^{s-\sigma}} &\le C
(\|\varphi\|_{H^s} + T^{1 - \frac 2q}\|u\|_{L_T^\infty H^s}
\|u\|_{L_T^q H_r^{s-\sigma}}^2)\\
&\le C(\|\varphi\|_{H^s} + T^{1 - \frac 2q}\rho^3)
\end{align*}
for some constant $C$. Here we used the conventional embedding that
if $2(s-\sigma) \ge n - \delta$ then $H_r^{s-\sigma} \hookrightarrow
L^{r_1}$ for any $r_1 \ge r$. Thus if we choose $\rho$ and $T$ so
that $C\|\varphi\|_{H^s} \le \frac \rho2$ and $CT^{1 - \frac
2q}\rho^3 \le \frac \rho2$, then we conclude that $\mathcal N$ maps from
$Y(T, \rho)$ to itself.

For any $u, v \in Y(T, \rho)$, we have
\begin{align}\begin{aligned}\label{diff}
d_Y(\mathcal N(u),\; &\mathcal N(v))\\
 &\lesssim \|F(u) - F(v)\|_{L_T^1H^s}\\
&\lesssim \|K_\gamma(|u|^2 - |v|^2)u\|_{L_T^1H^s} +
\|K_\gamma(|v|^2)(u - v)\|_{L_T^1H^s}.
\end{aligned}\end{align}
By  Lemma \ref{frac bound} and H\"{o}lder's inequality, we have for
sufficiently small $\varepsilon > 0$
\begin{align}\begin{aligned}\label{map3}
\|K_\gamma&(|u|^2 - |v|^2)u\|_{L_T^1H^s}\\
&\lesssim \|K_\gamma(|u|^2-|v|^2)\|_{L_T^2L^\infty}\|u\|_{L_T^\infty
H^s}\\
&\qquad\qquad +
\|K_\gamma(|u|^2-|v|^2)\|_{L_T^2H_\frac{2n}{\gamma+\varepsilon}^s}\|u\|_{L_T^2L^\frac{2n}{n-(\gamma
+\varepsilon)}}\\
&\lesssim \rho\||u|^2 -
|v|^2\|_{L_T^1L^{\frac{n}{n-(\gamma+\varepsilon)}}}^\frac12\||u|^2-|v|^2\|_{L_T^1L^\frac{n}{n-(\gamma-\varepsilon)}}^\frac12\\
&\qquad + \rho\|u-v\|_{L_T^\infty
H^s}(\|u\|_{L_T^2L^\frac{2n}{n-(\gamma-\varepsilon)}} +
\|v\|_{L_T^2L^\frac{2n}{n-(\gamma-\varepsilon)}})\\
&\qquad + \rho\|u -
v\|_{L_T^2L^\frac{2n}{n-(\gamma-\varepsilon)}}(\|u\|_{L_T^\infty
H^s} + \|v\|_{L_T^\infty H^s}).
\end{aligned}\end{align}
Now by another H\"{o}lder's inequality with respect to the time
variable, we have
$$
\|K_\gamma(|u|^2 - |v|^2)u\|_{L_T^1H^s} \lesssim T^{1 - \frac
2q}\rho^2 d_Y(u, v).
$$
Similarly,
\begin{align}\begin{aligned}\label{map4}
\|K_\gamma(|v|^2)&(u - v)\|_{L_T^1H^s}\\
&\lesssim \|K_\gamma(|v|^2)\|_{L_T^1L^\infty}\|u-v\|_{L_T^\infty
H^s}\\
&\qquad +
\|K_\gamma(|v|^2)\|_{L_T^2H_{\frac{2n}{\gamma+\varepsilon}}^s}\|u-v\|_{L_T^2L^{\frac{2n}{n-(\gamma+\varepsilon)}}}\\
&\lesssim
\|v\|_{L_T^2L^\frac{2n}{n-(\gamma-\varepsilon)}}\|v\|_{L_T^2L^\frac{2n}{n-(\gamma+\varepsilon)}}d_T(u,
v)\\
&\qquad + \|v\|_{L_T^\infty
H^s}\|v\|_{L_T^2L^\frac{2n}{n-(\gamma-\varepsilon)}}\|u-v\|_{L_T^2L^{\frac{2n}{n-(\gamma+\varepsilon)}}}.
\end{aligned}\end{align}
Thus we get
$$
\|K_\gamma(|v|^2)(u - v)\|_{L_T^1H^s} \lesssim T^{1 - \frac 2q}
\rho^2 d_Y(u, v).
$$
Substituting these two estimates into \eqref{diff} and then using
the fact $CT^{1-\frac2q}\rho^2 \le \frac12$ for small $T$, we
conclude that $\mathcal N$ is a contraction mapping.
\end{proof}


\bigskip

Now we show the local solutions can be extended globally in time by
using the energy conservation law. We first consider defocusing case.
\begin{theorem}\label{global2}
Let $m \ge 0$, $0 < \gamma < \min(2\alpha, n)$, $n \ge 1$. If $\lambda = +1$, then for any $\varphi \in H^\frac\alpha2$, then \eqref{int eqn} has
a unique solution $u \in C([0,\infty); H^\frac\alpha2) \cap L_{loc}^q
(H_r^{\frac\alpha2-\sigma})$, where $q = \frac{4}{\delta}$, $r =
\frac{2n}{n-\delta}$ and $\sigma = \frac{\delta(2-\alpha)}{4}$ for some $\delta$ with $0 < \delta < \min(\gamma, 2)$ and $\frac\alpha2 > \frac\gamma2 - \frac{\delta\alpha}4$.
\end{theorem}
\begin{proof}
Let $T^*$ be the maximal existence time. We will prove that $T^*$ is
infinite by contradiction. Suppose that $T^* < \infty$. Then the
local theory shows that $\|u\|_{L_{T^*}^qH_r^{\frac\alpha2 - \sigma}} =
\infty$. Since $\gamma < 2\alpha$, from the local existence Proposition
\ref{local2}, we see that the energy conservation law \eqref{consv}
holds. Thus if $\lambda = +1$, then at any $t < T^*$, the solution $u$ satisfies that
$$\frac12\|u(t)\|_{H^\frac\alpha2}^2 \le \frac12 \|u(t)\|_{L^2}^2 + E(u)
= \frac12\|\varphi\|_{L^2}^2 + E(\varphi).$$

From the estimate \eqref{map2} which is
used with $s = \frac \alpha2$, we have
$$
\|u\|_{L_T^q H_r^{\frac\alpha2 - \sigma}} \lesssim \|\varphi\|_{L^2}^2 +
E(\varphi) + T^{1 - \frac2q}(\|\varphi\|_{L^2}^2 +
E(\varphi))^\frac12\|u\|_{L_T^q H_r^{\frac\alpha2 - \sigma}}^2.
$$
Thus for sufficiently small $T$ depending on $\|\varphi\|_{L^2}^2 +
E(\varphi)$,
$$
\|u\|_{L^q(T_{j-1}, T_j; H_r^{\frac\alpha2 - \sigma})} \le
C(\|\varphi\|_{L^2}^2 + E(\varphi)),
$$ where $T_j - T_{j-1} = T$ for $j \le k-1$ and $T_k = T^*$
This means that $$\|u\|_{L^q(0,T^*; H_r^{\frac\alpha2 - \sigma})}^q \le
\sum_{1 \le j \le k}\|u\|_{L^q(T_{j-1}, T_j; H_r^{\frac12 -
\sigma})}^q \le (kC(\|\varphi\|_{L^2}^2 + E(\varphi)))^q <
\infty.$$ This is the contradiction to the hypothesis $T^* <
\infty$. This completes the
proof of Theorem \ref{global2}.
\end{proof}
\medskip

To treat the focusing problem we need more elaboration. Let us first
observe that for any $f \in H^\frac\alpha2$
\begin{align*}
|V(f)| &\le
\|\psi\|_{L^\infty}\||x|^{-\gamma}*|f|^2\|_{L^\frac{r}{r-2}}\|f\|_{L^2}^2
\le \|\psi\|_{L^\infty}\|f\|_{L^{\widetilde
r}}^2\|f\|_{L^r}^2,
\end{align*}
where $\frac1{\widetilde r} = 1- \frac1r -\frac\gamma{2n}$. If $\alpha < \gamma < 2\alpha$ and $2 < r < \frac{2n}{n-\alpha}$, then $2 < \widetilde r < \frac{2n}{n-\alpha}$. Thus from Sobolev embedding it follows that
\begin{align}\label{a priori potential}
|V(f)| \lesssim \|\psi\|_{L^\infty}\|f\|_{L^2}^{2(2-\frac\gamma\alpha)}\|f\|_{\dot
H^\frac\alpha2}^\frac{2\gamma}\alpha.
\end{align}

From \eqref{a priori potential} we can treat a variational problem. Let us invoke from \cite{chooz-ccm} that the embedding
$H_{rad}^\frac\alpha2 \hookrightarrow L^r$ is compact if $n \ge 2$, $1 < \alpha < 2$ and $2 < r < \frac{2n}{n-\alpha}$. Here $H_{rad}^\frac\alpha2$ is the Sobolev space $H^\frac\alpha2$ of radial functions. From this one can easily get the existence of nontrivial radial solution in $H^\frac\alpha2$ to the problem
$$J = \sup_{u \in H^\frac\alpha2 \setminus \{0\}} \frac{|V(u)|}{\|u\|_{\dot H^\frac\alpha2}^{\frac{2\gamma}\alpha} \|u\|_{L^2}^{2(2-\frac\gamma\alpha)}}.$$

Now we consider the focusing case.

\begin{theorem}\label{global3}
Let $m \ge 0$, $\lambda = -1$, $\alpha < \gamma < \min(2\alpha, n)$ and $n \ge 2$. If $\varphi \in H^\frac\alpha2$ and $\|\varphi\|_{\dot H^\frac\alpha2}$ is sufficiently small, \eqref{int eqn} has a unique solution $u \in C([0,\infty); H^\frac\alpha2) \cap L_{loc}^q
(H_r^{\frac\alpha2-\sigma})$, where $q = \frac{4}{\delta}$, $r =
\frac{2n}{n-\delta}$ and $\sigma = \frac{\delta(2-\alpha)}{4}$ for some $\delta$ with $0 < \delta < \min(\gamma, 2)$ and $\frac\alpha2 > \frac\gamma2 - \frac{\delta\alpha}4$.
\end{theorem}
\begin{proof}
From \eqref{a priori potential} we deduce that
$|E(\varphi)| = O(\|\varphi\|_{\dot H^\frac\alpha2}^2)$ as $\|\varphi\|_{\dot H^\frac\alpha2} \to 0$.
Thus we have
\begin{align*}
E(\varphi) = E(u) \ge \frac12\|u\|_{\dot H^\frac\alpha2}^2 - J\|u\|_{L^2}^{2(2-\frac\gamma\alpha)}\|u\|_{\dot H^\frac\alpha2}^\frac{2\gamma}\alpha.
\end{align*}
By the continuity argument we see that for any $\varphi$ with sufficiently small $\|\varphi\|_{\dot H^\frac\alpha2}$ such as $$|E(\varphi)| < 4^{-\frac\gamma{\gamma-\alpha}}\left(J\|\varphi\|_{L^2}^{2(2-\frac\gamma\alpha)}\right)^{-\frac\alpha{\gamma-\alpha}},$$
the corresponding solution $u$ satisfies the estimate $$\|u\|_{\dot H^\frac\alpha2}^2 \le 4|E(\varphi)|.$$
Then the conclusion follows in the same way as in the proof of Theorem \ref{global2}.
\end{proof}

Now we consider the small data global existence and scattering for
$2\alpha \le \gamma < n$.

\begin{theorem}\label{global4}
Let $m \ge 0$, $2\alpha \le \gamma < n$, $n \ge 3$ and $s > \frac\gamma2 -
\frac{\alpha}{2}$. Then there exists $\rho > 0$ such that for any
$\varphi \in H^s$ with $\|\varphi\|_{H^s} \le \rho$, \eqref{int eqn}
has a unique solution $u \in C_b([0, \infty); H^s) \cap L^2
(0, \infty; H_\frac{2n}{n-2}^{s-\frac{2-\alpha}{2}})$. Moreover there is $\varphi^+ \in H^s$ such that
$$
\|u(t) - U(t)\varphi^+\|_{H^s} \to 0\;\;{\rm as}\;\;t \to \infty.
$$
\end{theorem}
\begin{proof}
Let us define a complete metric space $(Y(\rho), d)$ with metric
$d_Y$ by
\begin{align*}
&Y(\rho) = \left\{v \in Y \equiv C_b ([0,\infty); H^s) \cap L^2(0,\infty; H_\frac{2n}{n-2}^{s-\frac{2-\alpha}2}) :
\|v\|_Y \le \rho\right\},\\
&\qquad\qquad\qquad\qquad\qquad \;\;d_Y(u, v) = \|u - v\|_Y.
\end{align*}
Then from the estimate \eqref{map2}, we have
$$
\|\mathcal N(u)\|_{Y} \le C\|\varphi\|_{H^s}
+ C\|u\|_{L^2(0, \infty; H_\frac{2n}{n-2}^{s-\frac{2-\alpha}2})}^2\|u\|_{L^\infty(0,\infty; H^s)}.
$$
If we choose sufficiently small $\rho$ such that $C\|\varphi\|_{H^s}
\le \frac\rho2$ and $C\rho^3 \le \frac \rho2$, then $\mathcal N$ maps
$Y(\rho)$ to itself. Similarly, from \eqref{diff}--\eqref{map4},
one can show that $d(\mathcal N(u), \mathcal N(v)) \le \frac12 d(u,v)$. This proves
the existence part.

To prove the scattering, let us define a function $\varphi^+$ by
$$
\varphi^+ = \varphi - i\int_0^\infty U(-t')F(u)(t')\,dt'.
$$
Then since the solution $u$ is in $Y(\rho)$, $\varphi^+ \in H^s$,
and therefore
\begin{align*}
\|u(t) - u^+(t)\|_{H^s} &\lesssim \int_t^\infty \|F(u)\|_{H^s}\,dt'\\
&\lesssim \|u\|_{L^\infty(0, \infty; H^s)}\int_t^\infty
\|u\|_{H_\frac{2n}{n-2}^{s-\frac{2-\alpha}2}}^2\,dt' \to 0\;\;\mbox{as}\;\;t\to \infty.
\end{align*}
\end{proof}

\section{Existence III: radial case}
In this section we establish the global existence theory of radial
solution of \eqref{main eqn} without regularity loss. We denote the
Banach space $X$ of radial functions by $X_{rad}$. We always assume
that $m \ge 0$ and $\psi$ is radially symmetric.

\subsection{Subcritical case}
We first consider the mass-( and energy-)subcritical problems.
\begin{theorem}\label{rad thm1}
$(1)$ Let $\frac{2n}{2n-1} \le \alpha < 2$ and $0 < \gamma <
\alpha$. If $\varphi \in L_{rad}^2$, then there exists a unique
solution $u$ of \eqref{main eqn} such that $u \in C_b([0, \infty) ;
L_{rad}^2) \cap L_{loc}^\frac{3\alpha}\gamma(0, \infty ;
L^\frac{2n}{n-\frac{2\gamma}3})$.

\noindent $(2)$ Let $\frac{2n}{2n-1} \le \alpha < 2$ and $\alpha <
\gamma < \min(2\alpha, n)$. If $\varphi \in H_{rad}^\frac\alpha2$
$(\|\varphi\|_{\dot H^\frac\alpha2}$ is sufficiently small if
$\lambda = -1)$, then there exists a unique solution $u$ of
\eqref{main eqn} such that $u \in C_b([0, \infty) ;
H_{rad}^\frac\alpha2) \cap L_{loc}^\frac{3\alpha}{\gamma-\alpha}(0, \infty
; H_\frac{2n}{n-\frac{2(\gamma-\alpha)}3}^\frac\alpha2)$.
\end{theorem}
Contrary to Theorems \ref{global2} and \ref{global3}, the mass-critical case is treated in the part (1) and a better Strichartz norm is obtained in the energy-subcritical case, part (2).

\begin{proof}
\textbf{Case (1)}. Let us define a complete metric space $(Z(T, \rho), d_Z)$ with metric $d_Z$ by
\begin{align*}
&Z(T, \rho) = \left\{v \in Z \equiv C_b ([0,T]; L_{rad}^2) \cap L_T^\frac{3\alpha}{\gamma} L^\frac{2n}{n-\frac{2\gamma}3}) :
\|v\|_Z \le \rho\right\},\\
&\qquad\qquad\qquad\qquad\qquad \;\;d_Z(u, v) = \|u - v\|_Z.
\end{align*}
For some $T$ and $\rho$ we will show that the mapping $\mathcal N$ is a contraction on $Z(T, \rho)$.

From \eqref{rad str1} and \eqref{rad str2} with $\theta = 0$ and
$(q_1,r_1) = (\frac{3\alpha}\gamma, \frac{2n}{n-\frac{3\gamma}2})$,
$(q_2, r_2) = (\infty, 2)$ (thus $1 - \frac1{r_2} = \frac3{r_1} -
\frac{n-\gamma}{n}$) we have for any $u \in Z(T, \rho)$
\begin{align*}
\|\mathcal N(u)\|_Z &\lesssim \|\varphi\|_{L^2} + \|K_\gamma(|u|^2) u\|_{L_T^1L^2} \lesssim \|\varphi\|_{L^2} + \|u\|_{L_T^3L^{r_1}}^3\\
&\lesssim \|\varphi\|_{L^2} + T^{1 - \frac\gamma\alpha} \|u\|_{L_T^{q_1}L^{r_1}}^3 \lesssim \|\varphi\|_{L^2}+ T^{1 - \frac\gamma\alpha}\rho^3.
\end{align*}
The involved constant is uniform on $m$ if $0 \le m \le m_0$. From the gap condition it follows that $\frac{2n}{2n-1} \le \alpha < 2$.

Similarly one can easily show that for any $u, v \in Z(T, \rho)$
\begin{align*}
d_Z(\mathcal N(u), \mathcal N(v)) \lesssim T^{1-\frac\gamma\alpha}\rho^2d_Z(u, v).
\end{align*}
For suitable $\rho$ and $T$, $\mathcal N$ becomes a contraction mapping, which means there is a unique solution $u \in Z(T, \rho)$. Now by the $L^2$ conservation and time iteration scheme, $u$ turns out to be a global solution of \eqref{main eqn}.

\textbf{Case (2)}. In this case we define the metric space $(Z(T, \rho), d_Z)$ by
\begin{align*}
&Z(T, \rho) = \left\{v \in Z \equiv C_b ([0,T]; H_{rad}^\frac\alpha2) \cap L_T^\frac{3\alpha}{\gamma-\alpha} H_\frac{2n}{n-\frac{2(\gamma-\alpha)}3}^\frac\alpha2) :
\|v\|_Z \le \rho\right\},\\
&\qquad\qquad\qquad\qquad\qquad \;\;d_Z(u, v) = \|u - v\|_Z.
\end{align*}
As above we choose $\theta = 0$, $(q_1, r_1) =
(\frac{3\alpha}{\gamma-\alpha},
\frac{2n}{n-\frac{2(\gamma-\alpha)}3})$ and $(q_2, r_2) = (\infty,
2)$. Then $$1-\frac1{r_2} = 2(\frac1{r_1} - \frac\alpha{2n}) -
\frac{n-\gamma}n + \frac1{r_1}$$ and we have
\begin{align*}
\|\mathcal N(u)\|_Z &\lesssim \|\varphi\|_{H^\frac\alpha2} + \|K_\gamma(|u|^2) u\|_{L_T^1H^\frac\alpha2} \\
&\lesssim \|\varphi\|_{L^2} + T^{2 - \frac\gamma\alpha} \|u\|_{L_T^{q_1}H_{r_1}^\frac\alpha2}^3\\
& \lesssim \|\varphi\|_{L^2}+ T^{2 - \frac\gamma\alpha}\rho^3
\end{align*}
and
\begin{align*}
d_Z(\mathcal N(u), \mathcal N(v)) \lesssim \|K_\gamma(|u|^2) u -
K_\gamma(|v|^2)v\|_{L_T^1H^\frac\alpha2} \lesssim
T^{2-\frac\gamma\alpha}\rho^2d_Z(u, v).
\end{align*}
We now have only to choose $T, \rho$ for contraction of $\mathcal N$. This yields the local existence.

Using energy conservation and time iteration scheme for $\lambda = +1$ and smallness argument as in Theorem \ref{global3} for $\lambda = -1$, we get a unique global solution. This completes the proof of Theorem \ref{rad thm1}.

\end{proof}

\subsection{Critical case}
\begin{theorem}\label{rad thm2}
$(1)$ Let $\frac{2n}{2n-1} \le \alpha < 2$ and $\alpha \le \gamma < n$. If $\varphi \in H_{rad}^\frac{\gamma-\alpha}2$ and $\|\varphi\|_{H^\frac{\gamma-\alpha}2}$ is sufficiently small, then there exists a unique solution $u$ of \eqref{main eqn}
such that $u \in C_b([0, \infty) ; H_{rad}^\frac{\gamma-\alpha}2) \cap L^3(0, \infty ; H_\frac{2n}{n-\frac{2\alpha}3}^\frac{\gamma-\alpha}2)$.

\noindent $(2)$ Let $\frac{2n}{2n-1} \le \alpha < 2$ and $\frac\alpha3 \le \gamma < \alpha$. If $\varphi \in \dot H_{rad}^\frac{\gamma-\alpha}2$ and $\|\varphi\|_{\dot H^\frac{\gamma-\alpha}2}$ is sufficiently small, then there exists a unique solution $u$ of \eqref{main eqn}
such that $u \in C_b([0, \infty) ; H_{rad}^\frac{\gamma-\alpha}2) \cap L^3(0, \infty ; L^\frac{2n}{n-(\gamma - \frac\alpha3)})$.
\end{theorem}
\begin{proof}
\textbf{Case (1)}. We define the metric space $(Z(\rho), d_Z)$ by
\begin{align*}
&Z(\rho) = \left\{v \in Z \equiv C_b ([0,\infty); H_{rad}^\frac{\gamma-\alpha}2) \cap L^3(0,\infty; L^\frac{2n}{n-(\gamma-\frac\alpha3)}) :
\|v\|_Z \le \rho\right\},\\
&\qquad\qquad\qquad\qquad\qquad \;\;d_Z(u, v) = \|u - v\|_Z.
\end{align*}
By the same way as the part (2) of Theorem \ref{rad thm2} we choose
$\theta = 0$, $(q_1, r_1) = (3, \frac{2n}{n-\frac{2\alpha}3})$ and
$(q_2, r_2) = (\infty, 2)$ so that $$1-\frac1{r_2} = 2(\frac1{r_1} -
\frac{\gamma-\alpha}{2n}) - \frac{n-\gamma}n + \frac1{r_1}.$$ Then
we have
\begin{align*}
\|\mathcal N(u)\|_Z &\lesssim \|\varphi\|_{H^\frac{\gamma-\alpha}2} + \|K_\gamma(|u|^2) u\|_{L_T^1H^\frac{\gamma-\alpha}2} \\
&\lesssim \|\varphi\|_{H^\frac{\gamma-\alpha}2} + \|u\|_{L_T^{q_1}H_{r_1}^\frac{\gamma-\alpha}2}^3\\
& \lesssim \|\varphi\|_{L^2}+ \rho^3
\end{align*}
and also
\begin{align*}
d_Z(\mathcal N(u), \mathcal N(v)) \lesssim  \rho^2d_Z(u, v).
\end{align*}
If $C\|\varphi\|_{H^\frac{\gamma-\alpha}{2}} \le \frac\rho2$ and $C\rho^2 \le \frac12$, then $\mathcal N$ is a contraction.
\medskip

\textbf{Case (2)}.
Take metric space $Z(\rho)$ as
\begin{align*}
&Z(\rho) = \left\{v \in Z \equiv C_b ([0,\infty); \dot H_{rad}^\frac{\gamma-\alpha}2) \cap L^3(0,\infty; L^\frac{2n}{n-(\gamma-\frac\alpha3)}) :
\|v\|_Z \le \rho\right\},\\
&\qquad\qquad\qquad\qquad\qquad \;\;d_Z(u, v) = \|u - v\|_Z.
\end{align*}
Then it follows from \eqref{rad str1} and \eqref{rad str2} with
$\theta = \frac{\gamma-\alpha}2$, $(q_1, r_1) = (3,
\frac{2n}{n-(\gamma-\frac\alpha3)})$ and $(q_2, r_2) = (\infty,
\frac{2n}{n-(\alpha-\gamma)})$ that for any $u \in Z(\rho)$
\begin{align*}
\|\mathcal N(u)\|_Z &\lesssim \|\varphi\|_{\dot
H^\frac{\gamma-\alpha}2} + \int_0^\infty\|K_\gamma(|u|^2) u\|_{\dot
H^\frac{\gamma-\alpha}2 \cap L^\frac{2n}{n+\alpha-\gamma}}\,dt.
\end{align*}
Since $\|\psi\|_{\dot H^\frac{\gamma-\alpha}2} \lesssim \|\psi\|_{L^{r_2'}}$ and $\frac1{r_2'} = \frac2{r_1} - \frac{n-\gamma}{n} + \frac1{r_1}$,
$$
\|\mathcal N(u)\|_Z  \lesssim \|\varphi\|_{\dot H^\frac{\gamma-\alpha}2} + \rho^3
$$
and for any $u, v \in Z(\rho)$
$$
d_Z(\mathcal N(u), \mathcal N(v)) \lesssim \rho^2d_Z(u, v).
$$
Taking small $\|\varphi\|_{\dot H^\frac{\gamma-\alpha}2}$ and $\rho$ completes the proof of (2) of Theorem \ref{rad thm2}.
\end{proof}

\section{Existence IV: via weighted Strichartz estimates}
In this section we show the global well-posedness below $L^2$. To
avoid complexity we only consider the case $m = 0$ and $n = 3$. We
utilize the weighted Strichartz estimates \eqref{lomegainfty} and
\eqref{lomega2} and have the following.
\begin{theorem}\label{belowl2thm}
Let $\psi \in L_{rad}^\infty$ and $m = 0$. Suppose that $n = 3$, $\frac {21 + \sqrt{21}}{15} < \alpha \leq 2$
and $\frac {15\alpha - \alpha^2}{12 + 2\alpha} < \gamma < \alpha$.
Then there exists a positive constant $\rho$ depending on $n,
\alpha, \gamma$ and $\lambda$ such that if $\varphi \in
\dot{H}^{s_c}H_\omega^{s_1 + s_2}$ and
$\||\nabla|^{s_c}d_\omega^{s_1+s_2}\varphi\|_{L_x^2} < \rho$,
then the integral equation \eqref{int eqn} has a unique solution
$u\in C_b([0, \infty);\dot{H}^{s_c}H_\omega^{s_1+s_2})$,
where $s_1 = \frac 2q - \frac {\gamma + 1 - \alpha}{2}$ and $s_2$
satisfies that
\begin{align*}
\max\left(\frac {n+1}{q_1} - \frac {\alpha}2, \gamma + 3 - \alpha +
\frac {n+1}{q_1} - \frac {4}{q} \right) < s_2 < \min\left(\frac
{n-1}q, \gamma - \frac {n}{q_1}\right).
\end{align*}
Moreover, there exists $\varphi_+ \in \dot H^{s_c}H^{s_1 + s_2}_\omega$ such that
$$\|u(t) - U(t)\varphi_{+}\|_{\dot H^{s_c}H^{s_1 + s_2}_\omega} \rightarrow 0 \text{ as } t\rightarrow \infty.$$
\end{theorem}

The proof of the theorem consists of several subsections.

\subsection{Weighted estimates}
In this subsection we assume that $n \ge 2$. We introduce several weighted estimates based on the Strichartz estimates \eqref{lomegainfty} and \eqref{lomega2}. In fact, from interpolation of \eqref{lomegainfty} and \eqref{lomega2} we get the following:
\begin{lemma}\label{interpolation}
Let $n \ge 2$ and $2 \leq q \le \infty$. Then\\
$(1)$ For each $c$ and $\delta_1$ such
that
\begin{align*}
&-\frac nq < c < - \frac nq + \frac{n-1}2,\\
&\quad  \delta_1 \le -\frac nq + \frac{n-1}2 - c,
\end{align*}
we have
\begin{align}\label{inter1}
\| |x|^c|\nabla|^{c+\frac{n+\alpha}q - \frac
n2}d_\omega^{\delta_1}U(t)\varphi\|_{L^q_tL^q_rL^2_\omega} \lesssim
\|\varphi\|_{L_x^2}. \end{align}

$(2)$ For each $c$ with $-\frac nq < c < -\frac 1q$ and $\delta_2
\le -c - \frac 1q$ we have
\begin{align}\label{inter2}
\|
|x|^c|\nabla|^{c+\frac{\alpha}{q}}d_\omega^{\delta_2}U(t)\varphi\|_{L_t^q
L_x^2} \lesssim \|\varphi\|_{L_x^2}. \end{align}
\end{lemma}

\begin{proof}
Interpolating \eqref{lomegainfty} and \eqref{lomega2}, we obtain
\eqref{inter1} after arranging interpolation indices with respect to
$c$ of interpolated weight $|x|^c$.
 For \eqref{inter2} one can use \eqref{lomega2} and trivial estimate $\|U(t)\varphi\|_{L_t^\infty L_x^2} = \|\varphi\|_{L_x^2}$.
\end{proof}

To handle the Hartree nonlinearity we consider the following
weighted convolution estimates ( see \cite{chozsash} and \cite{chonak}).

\begin{lemma}\label{weighted con}
Let $1 \le p, q \le \infty$, $0\le d_1 < d_2 < \frac{n-1}{p'}$ and
$\frac 1{q} \le 1- \frac{d_2}{n-1}$. Then we have
\begin{align}\label{weighted1}
 \||x|^{d_1}(|x|^{-\frac np-d_2}*f)\|_{L^p_x} \lesssim \||x|^{-(d_2-d_1)}f\|_{L^1_r
 L^{q,\,1}_\omega}.
 \end{align}
Moreover, if $p = \infty$, then $d_1 = d_2$ is also allowed.
Here $L^{q,\,1}_\omega$ denotes the Lorentz space on the unit sphere.
\end{lemma}

 Throughout the section the triplet $(c_0, c_1, c_2)$
denotes
\begin{center}$\left(\frac{\gamma + n - \alpha}{2}
- \frac{n+\alpha}q, \quad \frac {n+\alpha}{q_1} - \alpha, \quad
\frac {n+\alpha}{q_2} + \frac{\gamma - n -
\alpha}{2}\right).$\end{center} Here we use the explicit exponents
\begin{align*}
\frac 1{q_1} &= \frac12(\frac {\alpha + \gamma}{\alpha - 1} + \frac {\alpha + 2\gamma}{4n + 2}),\\
\frac 1{q_2} &= \frac 12(\frac {\alpha - \gamma + 1}{2\alpha} + \frac {\alpha + 1 - \gamma + \frac {2n - 4}{q1}}{4}),\\
\frac 1q &= 1 - \frac 1{q_1} - \frac 1{q_2}.
\end{align*}
Note that $c_0 = c_1 + c_2$.

\subsection{Duhamel formula}
One can verify that $q,q_1$ and $q_2$ defined above satisfy all the assumption in the
following lemmas.

We first consider $\dot H^{s_c}H_\omega^{s_1+s_2}$ estimate for the
Duhamel part $U(t)\Phi_t$, where $$\Phi_t \equiv -i\lambda\int_0^t U(-t')K_\gamma(|u|^2)u(t')\,dt'.$$
\begin{lemma}\label{general duha1}
Let $s_1 = \frac 2q - \frac{\gamma + 1 - \alpha}2$ and $0 \leq s_2
\leq \min(\gamma - \frac n{q_1}, \frac{n-1}{q_1})$. Suppose that
$q_1$ satisfies $\frac {\alpha - \gamma}{\alpha} < \frac 1{q_1} \leq
\frac {\alpha}{n + \alpha}$, then we have
\begin{align*}
\| |\nabla|^{s_c}d_\omega^{s_1 + s_2}U(t)\Phi_t\|_{L^\infty_tL^2_x}
\lesssim
\||x|^{-c_0}d_\omega^{s_2}[K_\gamma(|u|^2)u]\|_{L^{q'}_tL^{q'}_rL^2_\omega}
\lesssim \widetilde W_1(u)^2\widetilde W_2(u),
\end{align*}
where
\begin{align*}
&\widetilde W_1(u) = \||x|^{-(\gamma - \frac n{q_1} +
c_1)/2}d_\omega^{(\gamma - \frac n{q_1} +
s_2)/2}u\|_{L^{2q_1}_tL^2_x},\\
&\widetilde W_2(u) = \||x|^{-c_2}d_\omega^{\frac
{n-1}{q_1}}u\|_{L^{q_2}_tL^{q_2}_rL^2_\omega}.
\end{align*}
\end{lemma}

\begin{proof}
By the dual estimate of \eqref{inter1} and Strichartz estimate
\eqref{homo str} we have
$$
\|\int_0^t U(-t') K_\gamma(|u|^2)u(t')\,dt'\|_{L_t^\infty L^2_x}
\lesssim
\||x|^{-c_0}|\nabla|^{-s_c}d_\omega^{-s_1}(K_\gamma(|u|^2)u)\|_{L^{q'}_tL^{q'}_rL^2_\omega},
$$
which implies
$$\||\nabla|^{s_c}d_\omega^{s_1 + s_2}U(t)\Phi_{t}\|_{L^\infty_t L^2_x} \lesssim \||x|^{-c_0}d_\omega^{s_2}(K_\gamma(|u|^2)u)\|_{L^{q'}_tL^{q'}_rL^2_\omega}.$$
Since $d_\omega^{s_2}$ commutes with radial function $\psi$ and $|x|^{-c_0}$, we obtain
$$
\||x|^{-c_0}d_\omega^{s_2}[(|x|^{-\gamma}*|u|^2)u]\|_{L^{q'}_tL^{q'}_rL^2_\omega}
\lesssim \|d_\omega^{s_2}[|x|^{-c_1}(|x|^{-\gamma}*|u|^2)
|x|^{-c_2}u]\|_{L^{q'}_tL^{q'}_rL^2_\omega}.
$$
Now by Leibniz rule on the unit sphere with $1/q' = 1/q_1 + 1/q_2 $
\begin{align}\begin{aligned}\label{leib1}
&\||x|^{-c_0}d_\omega^{s_2}[(|x|^{-\gamma}*|u|^2)u]\|_{L^{q'}_tL^{q'}_rL^2_\omega}\\
\lesssim\;&
\||x|^{-c_1}d_\omega^{s_2}(|x|^{-\gamma}*|u|^2)\|_{L^{q_1}_{t,x}}\||x|^{-c_2}d_\omega^{s_2}u\|_{L^{q_2}_tL^{q_2}_rL^{\widetilde{q_2}}_\omega},
\end{aligned}\end{align}
where $1/2 = 1/q_1 + 1/\widetilde{q_2} - s_2/(n-1)$. Here we need $0
\le s_2 \le \frac{n-1}{q_1}$. By using Sobolev imbedding on the unit
sphere, we obtain
\begin{align}\label{leib2}\||x|^{-c_2}d_\omega^{s_2}u\|_{L^{q_2}_tL^{q_2}_rL^{\widetilde{q_2}}_\omega}
\lesssim \||x|^{-c_2}d_\omega^{\frac
{n-1}{q_1}}u\|_{L^{q_2}_tL^{q_2}_rL^{2}_\omega}.\end{align}

Since $d_\omega^{s_2}$ also commutes with the convolution operator
$|x|^{-\gamma}*$, we have
$$\||x|^{-c_1}d_\omega^{s_2}(|x|^{-\gamma}*|u|^2)\|_{L^{q_1}_{x}} = \||x|^{-c_1}(|x|^{-\gamma}*(d_\omega^{s_2}(|u|^2)))\|_{L^{q_1}_{x}}\;\;\mbox{a.e.}\;\;t.$$
By using the weighted convolution estimate \eqref{weighted1}, we get
$$
\||x|^{-c_1}(|x|^{-\gamma}*(d_\omega^{s_2}(|u|^2)))\|_{L^{q_1}_{x}}\\
\lesssim
\||x|^{-\widetilde{c_1}}d_\omega^{s_2}(|u|^2)\|_{L^1_rL^{\frac{n-1}{n
- 1 - (\gamma - \frac n{q_1})},\,1}_\omega},
$$
where $\widetilde{c_1} = \gamma - \frac n{q_1} + c_1$. Since $s_2
\le \gamma - \frac{n}{q_1} < n-1 - \gamma + \frac{n}{q_1}$, the
Leibniz rule on the unit sphere gives
\begin{align*}
&\quad\, \left\||x|^{-\widetilde{c_1}}d_\omega^{s_2}(|u|^2)\right\|_{L^1_rL^{\frac{n-1}{n - 1 - (\gamma - \frac n{q_1})},\,1}_\omega}\\
&\lesssim
\left\||x|^{-\frac{\widetilde{c_1}}2}d_\omega^{s_2}u\right\|_{L^2_rL^{\frac{2(n-1)}{n
- 1 - (\gamma - \frac n{q_1} -
s_2)},\,2}_\omega}\left\||x|^{-\frac{\widetilde{c_1}}2}
u\right\|_{L^2_rL^{\frac{2(n-1)}{n - 1 - (\gamma - \frac n{q_1} +
s_2)},\,2}_\omega}.
\end{align*}
Using the Sobolev embedding on the sphere again, we obtain
\begin{align*}
\left\||x|^{-\widetilde{c_1}}d_\omega^{s_2}(|u|^2)\right\|_{L_t^{q_1}L^1_rL^{\frac{n-1}{n
- 1 - (\gamma - \frac n{q_1})},\,1}_\omega}\lesssim
\left\||x|^{-\frac{\widetilde{c_1}}2}d_\omega^{(\gamma - \frac
n{q_1} + s_2)/2}u\right\|_{L_t^{2q_1}L^2_x}^2.
\end{align*}
Combining this with \eqref{leib1} and \eqref{leib2}, we get the
desired estimate.
\end{proof}

If we further restrict the range of $q_1, q_2$, then we can handle
the weighted norms of \eqref{general duha1} in a closed form through the Christ-Kiselev lemma (for instance see
\cite{chriki, tao, ahncho}), which is stated as follows:
\begin{lemma}[Christ-Kiselev lemma]\label{chriki}
Let $1 \le r < q \leq \infty$, and $X, Y$ be Banach spaces. Suppose
that
$$\|U(t)\phi\|_{L^q_t(Y)} \lesssim \|\phi\|_{L_x^2}\;\; \text{ and }\;\; \|\int_0^\infty U(-t')g(t')dt'\|_{L_x^2} \lesssim \|g\|_{L^r_t(X)}.$$
Then
$$\|\int_0^t U(t-t')g(t')dt'\|_{L^q_t(Y)} \lesssim \|g\|_{L_t^r(X)}.$$
\end{lemma}
Now we consider weighted estimates for Duhamel part.

\begin{lemma}\label{general duha2}
Let $s_1 = \frac 2q - \frac{\gamma + 1 - \alpha}2$ and $\max(\gamma
- \frac{n+1}{q_1} + 3 - \alpha - \frac{4}{q}, \frac{n+1}{q_1} -
\frac {\alpha}2) \le s_2 \le \min(\gamma - \frac n{q_1},
\frac{n-1}{q_1})$. Suppose $\frac {\alpha - \gamma}{\alpha -1} <
\frac 1{q_1} \leq \frac {\alpha}{n + \alpha}$ and $\frac {\alpha -
\gamma + 1}{2\gamma} < \frac 1{q_2} \leq \frac 12$. Then we have
$$\widetilde W_1(U(t)\Phi_t) + \widetilde W_2(U(t)\Phi_t) \lesssim \widetilde W_1(u)^2\widetilde W_2(u).$$
\end{lemma}
\begin{proof}
From the dual estimates of \eqref{inter1} with $c = c_0$ it follows
that
\begin{align}\begin{aligned}\label{dual2}
\|\int_0^\infty U(-t') K_\gamma(|u|^2)u(t')\,dt'\|_{L^2_x}\lesssim
\||x|^{-c_0}|\nabla|^{-s_c}d_\omega^{-s_1}[K_\gamma(|u|^2)u]\|_{L^{q'}_tL^{q'}_rL^2_r}.
\end{aligned}\end{align}
Since $q' < q_2$, by Lemma \ref{chriki} together with, \eqref{inter1} with $c = -c_2$ and
\eqref{dual2} we have
\begin{align}\begin{aligned}\label{ck1}
&\quad\;  \||x|^{-c_2}|\nabla|^{-s_c}d_\omega^{\frac 2{q_2} + \frac{\gamma-3}2}U(t)\Phi_t\|_{L^{q_2}_tL^{q_2}_rL^2_\omega}\\
&\lesssim
\||x|^{-c_0}|\nabla|^{-s_c}d_\omega^{-s_1}[K_\gamma(|u|^2)u]\|_{L^{q'}_tL^{q'}_rL^2_r},
\end{aligned}\end{align}
which implies
\begin{align*}
\||x|^{-c_2}d_\omega^{\frac2{q_2} + \frac{\gamma-3}{2} + s_1 +
s_2}U(t)\Phi_t\|_{L^{q_2}_tL^{q_2}_rL^2_\omega} \lesssim
\||x|^{-c_0}d_\omega^{s_2}[K_\gamma(|u|^2)u]\|_{L^{q'}_tL^{q'}_rL^2_r}.
\end{align*}
Since $\frac{n-1}{q_1} \le \frac2{q_2} + \frac{\gamma-3}{2}  + s_1 +
s_2$, we get $\widetilde W_2(U(t)\Phi_t) \lesssim \widetilde
W_1(u)^2\widetilde W_2(u)$.

By a similar way to get \eqref{dual2} and \eqref{ck1} with the
estimates \eqref{inter2} instead of \eqref{inter1} we get
\begin{align*}
&\quad\;\||x|^{-(\gamma - \frac n{q_1} + c_1)/2}|\nabla|^{-s_c}d_\omega^{-(\frac{2-\gamma}2 -\frac{1}{2q_1})} U(t)\Phi_t\|_{L^{2q_1}_tL^2_x}\\
& \lesssim
\||x|^{-c_0}|\nabla|^{-s_c}d_\omega^{-s_1}[K_\gamma(|u|^2)u]\|_{L^{q'}_tL^{q'}_rL^2_\omega}.
\end{align*}
Then by angular regularity shift, we also have
\begin{align*}
&\quad\;\||x|^{-(\gamma - \frac n{q_1} + c_1)/2}d_\omega^{-(\frac{2-\gamma}2 -\frac{1}{2q_1}) + s_1 + s_2} U(t)\Phi_t\|_{L^{2q_1}_tL^2_x}\\
 &\lesssim \||x|^{-c_0}d_\omega^{s_2}[K_\gamma(|u|^2)u]\|_{L^{q'}_tL^{q'}_rL^2_\omega},
\end{align*}
which implies $\widetilde W_1(U(t)\Phi_t) \lesssim \widetilde
W_1(u)^2\widetilde W_2(u)$ because $(\gamma-\frac n{q_1} + s_2)/2 <
-(\frac{2-\gamma}2 -\frac{1}{2q_1}) + s_1 + s_2$ for $s_2$ as
stated. This completes the proof of Lemma \ref{general duha2}.
\end{proof}

We note that $\max(\frac {n+1}{q_1} - \frac {\alpha}2, \gamma + 3 -
\alpha + \frac {n+1}{q_1} - \frac {4}{q})$ is strictly less than
$\min(\frac {n-1}q, \gamma - \frac {n}{q_1})$. So, one can find a
common $s_2$ which meets the condition of Theorem $\ref{belowl2thm}$
and the requirements of Lemmas $\ref{general duha1}$, $\ref{general
duha2}$.

Now we are ready to prove Theorem \ref{belowl2thm}.
\subsection{Proof of Theorem \ref{belowl2thm}}
For $\varepsilon >
0$, let us define function space $B_\rho$ by
\begin{align*}
B_\rho \equiv \{ u \in C(\mathbb{R};\dot{H}^{s_c}H_\omega^{s_1 +
s_2})\ :\  \|u\|_B \le \rho\},
\end{align*}
where
\begin{align*}
\|u\|_B &= \||\nabla|^{s_c}d_\omega^{s_1+s_2}u\|_{L_t^\infty L_x^2}
+ \widetilde W_1(u) + \widetilde W_2(u).
\end{align*}
Then the set $B_\rho$ is a complete metric space endowed with
the metric
\begin{align*}
d_B(u,v)\ \equiv\
&\||\nabla|^{s_c}d_\omega^{s_1+s_2}(u-v)\|_{L_t^\infty L_x^2} +
\widetilde W_1(u-v) + \widetilde W_2(u-v).
\end{align*}
Now we define
$$\mathcal{N}(u) = U(t)(\varphi + \Phi_t) \;\;\mbox{on}\;\; B_\rho.$$
and show the mapping $\mathcal N$ is a contraction mapping
from $B_\rho$ to itself for a sufficiently small $\rho$.

First, from Lemma \ref{interpolation} it follows that
\begin{align}\begin{aligned}\label{linear-weighted-gl}
&\||\nabla|^{s_c}d_\omega^{s_1+s_2}U(t)\varphi\|_{L_t^\infty L_x^2} + \||x|^{-c_2}d_\omega^{\frac{n-1}{q_1}}U(t)\varphi\|_{L_t^{q_2}L_r^{q_2}L_\omega^2}\\ &+ \||x|^{-(-\frac{n}{q_1} + \gamma + c_1)/2}d_\omega^{(\gamma - \frac{n}{q_1} + s_2)/2}U(t)\varphi\|_{L_t^{2q_1}L_x^2} \lesssim \||\nabla|^{s_c}d_\omega^{s_1+s_2}\varphi\|_{L_x^2}.\\
\end{aligned}\end{align}
On the other hand, for any $u, v \in B_\rho$ we have for any $a,
\beta \in \mathbb R$
\begin{align*}
&\quad \left||x|^ad_\omega^\beta[|x|^{-\gamma}*(|u|^2)u)]  - |x|^ad_\omega^\beta[|x|^{-\gamma}*(|v|^2)v)]\right|\\
&\leq \left||x|^ad_\omega^\beta[|x|^{-\gamma}*(|u|^2)(u - v)]\right|\\
&\quad + \left||x|^ad_\omega^\beta[|x|^{-\gamma}*((u -
v)\bar{v})v]\right| +
\left||x|^ad_\omega^\beta[|x|^{-\gamma}*(u\overline{(u-v)})v]\right|.
\end{align*}
Then by adopting the arguments such as duality, Strichartz estimate,
and Christ-Kiselev lemma, as in the proofs of Lemmas \ref{general
duha1}, \ref{general duha2} we obtain the following.
\begin{align}\label{nonrad-contr}
d_B(\mathcal{N}(u), \mathcal{N}(v)) \lesssim (\widetilde W_1(u)
+ \widetilde W_2(u) + \widetilde W_1(v) + \widetilde W_2(v))^2d_B(u,
v).
\end{align}
Therefore
\begin{align}\label{nonlinear-weighted-gl}
&d_B(\mathcal{N}(u), \mathcal{N}(v)) \le C\rho^2 d_B(u, \,v)
\end{align}
for some constant $C$ independent of $u, v, \rho$.
 Now choose $\rho$ and the size of the norm $\|\varphi\|_{\dot
H^{s_c}H_\omega^{s_1+s_2}}$ small enough to ensure that $C\rho^2
\le \frac12$ and $C\|\varphi\|_{\dot H^{s_c}H_\omega^{s_1+s_2}} \le
\frac12 \rho$. Then combining \eqref{linear-weighted-gl} and
\eqref{nonlinear-weighted-gl}, we conclude that the mapping $\mathcal{N}$
becomes a contraction on $B_\rho$.

Now we show the existence of scattering. Let us define functions $\varphi_+$ by
$$\varphi_+ = \varphi - i\lambda\int_0^{\infty} U(-t')[K_\gamma(|u|^2)u](t')\,dt'.$$
Then by the estimates \eqref{nonrad-contr},  $\varphi_{\pm} \in \dot
H^{s_c}H^{s_1 + s_2}_\omega$ and
$$\|u(t) - U(t)\varphi_+\|_{\dot H^{s_c}H^{s_1 + s_2}_\omega} \rightarrow 0 \text{ as } t\rightarrow \infty.$$
This completes the proof of Theorem \ref{belowl2thm}.

\section{Finite time blowup}
In this section we consider the blowup dynamics of massive focusing mass critical FNLS
($m > 0$, $\gamma = \alpha, \lambda = -1$). For this purpose we
adapt the Virial type argument of \cite{frohlenz2}, in which the evolution of two quantities $\langle u, A u \rangle$ and $\langle u, M u\rangle $ for
$$
A = -\frac i2(\nabla \cdot x + x \cdot \nabla ), \quad M = x \cdot D_m^{2-\alpha} x.$$

It is obvious from Proposition \ref{local} that if $\varphi \in H^k,
k = \max(3, \frac\gamma2)$, then there exists a maximal existence
time $T^* > 0$ and a unique solution $u \in C([0, T^*); H^k) \cap
C^1([0,T^*); H^{k-1})$ of \eqref{main eqn}. If $T^* < \infty$, then
$\lim_{t\nearrow T^*} \|u(t)\|_{H^\frac\gamma2} = \infty$. If
further $x\varphi, |x|\nabla\varphi \in L^2$, then we can show the propagation of
moment: $xu(t), |x|\nabla u(t) \in L^2$ for all $t \in [0, T^*)$. We postpone
the proof to the end of this section.

Now let us introduce our
blowup result.

\begin{theorem}\label{blow}
Set $\gamma = \alpha$ and $m > 0$. Let $1 < \alpha < 2$ and $n \ge
4$. Suppose that $\psi$ is smooth radial function with $\psi'(\rho) = \partial_r \psi(\rho) \le 0$, $|\psi'(\rho)| \lesssim \frac1{\rho}$ for $\rho  > 0$, and $\varphi \in H_{rad}^k$ and $x\varphi, |x|\nabla\varphi \in L_{rad}^2$ with $E(\varphi) < 0$, we have that for each $m$ the maximal existence time $T_m^* \le r_m$ and $\lim_{t\nearrow T_m^*} \|u(t)\|_{H^\frac\gamma2} = \infty$, where $r_m $ is the positive root of
$$
2\alpha^2E(\varphi)t^2 + 2\alpha(\langle \varphi, A \varphi \rangle+ C\|\varphi\|_{L^2}^4)t + \langle \varphi, M \varphi\rangle.
$$
Here $C$ does not depend on $m$.
\end{theorem}

\subsection{Proof of Theorem \ref{blow}}

Let us now show the theorem. We begin with the dilation operator
$$
A = -\frac i2(\nabla \cdot x + x \cdot \nabla ).
$$
Since $u \in H^k$ and $xu, |x|\nabla u \in L^2$, $\langle u, A u\rangle$ is well-defined and so is
\begin{align}\label{commut0}
\frac{d}{dt}\langle u, A u\rangle = i\langle u, [H, A] u\rangle,
\end{align}
where $H = D_m^\alpha + \mathcal V$ and $\mathcal V = -
K_\alpha(|u|^2) = - (\psi/|\cdot|^\alpha) * |u|^2$. Here $[H, A]$
denotes the commutator $HA - AH$. As a matter of fact we have the
following.
\begin{lemma}\label{virial-lemma1}
Let $u$, $\varphi$ and $\psi$ be as above in Theorem \ref{blow}.
Then
\begin{align}\label{1st virial}
\frac{d}{dt}\langle u, A u\rangle \le 2\alpha E_2(\varphi).
\end{align}
\end{lemma}
\begin{proof}[Proof of Lemma \ref{virial-lemma1}]
Using the identity $D_m^\alpha x = xD_m^\alpha - \alpha
D_m^{\alpha-2} \nabla$, we have
\begin{align}\label{commut1}
[\,D_m^\alpha, A] = -i \alpha D_m^{\alpha-2}D_0^2.
\end{align}
Similarly,
\begin{align}\label{commut2}
[\,\mathcal V, A] = i (x \cdot \nabla) \mathcal V.
\end{align}
Substituting \eqref{commut1} and \eqref{commut2} into \eqref{commut0}, we get
\begin{align}\label{virial0}
\frac{d}{dt}\langle u, A u\rangle = \alpha\langle u, D_m^\alpha u \rangle - \alpha m^2\langle u, D_m^{\alpha-2}u \rangle - \langle u, (x \cdot \nabla)\mathcal V u\rangle.
\end{align}
For the second term on RHS of \eqref{virial0} we obtain the
following identities:
\begin{align*}
(x \cdot \nabla)\mathcal V &= \alpha \int \frac{\psi(|x-y|)}{|x-y|^{\alpha}}|u(y)|^2\,dy - \int \frac{\psi'(|x-y|)}{|x-y|^{\alpha}}|x-y||u(y)|^2\,dy\\
&\qquad+ \int \left(\alpha \frac{\psi(|x-y|)}{|x-y|^{\alpha+1}} -
\frac{\psi'(|x-y|)}{|x-y|^{\alpha}}\right)\frac{y\cdot
(x-y)}{|x-y|}|u(y)|^2\,dy,
\end{align*}
\begin{align*}
\langle u, (x \cdot \nabla)\mathcal V u\rangle &= - 4\alpha V(u) - \int\!\!\!\int \frac{|x-y|\psi'(|x-y|)}{|x-y|^{\alpha}}|u(x)|^2|u(y)|^2\,dxdy\\
&- \langle u, (x \cdot \nabla)\mathcal V u\rangle,
\end{align*}
which implies
\begin{align*}
\langle u, (x \cdot \nabla)\mathcal V u\rangle &= - 2\alpha V(u) -
\frac12\int\!\!\!\int \frac{|x-y|
\psi'(|x-y|)}{|x-y|^{\alpha}}|u(x)|^2|u(y)|^2\,dxdy.
\end{align*}
Substituting this into \eqref{virial0},
we have
\begin{align*}
\frac{d}{dt}\langle u, A u\rangle \le 2\alpha E(\varphi) +
\frac12\int\!\!\!\int\left(
|x-y|\psi_1'(|x-y|)\right)\frac{|u(x)|^2|u(y)|^2}{|x-y|^\alpha}\,dxdy.
\end{align*}
Since $\psi'(|x|) \le 0$, we get \eqref{1st virial}.
\end{proof}

Next we consider the nonnegative quantity $\langle u, M u \rangle$ with
$$
M \equiv x \cdot D_m^{2-\alpha} x = \sum_{k = 1}^n x_k D_m^{2-\alpha}x_k.
$$
From the regularity and decay condition of $u$ the quantity $\langle u(t), M u(t) \rangle$ is well-defined and finite for all $t \in [0,T^*)$ since $|\langle u, M u\rangle| \lesssim_m  \|xu\|_{L^2}(\|xu\|_{L^2}+ \|x\cdot \nabla u\|_{L^2})$, and so is
\begin{align}\label{virial1}
\frac{d}{dt}\langle u, M u \rangle = i\langle u, [H, M]u\rangle  = i\langle u, [\,D_m^\alpha, M] u\rangle -i\langle u, [K_\alpha(|u|^2), M]u\rangle.
\end{align}
We have the following.
\begin{lemma}\label{virial-lemma2}
With the same condition as in Theorem \ref{blow}, we have
\begin{align}\label{2nd virial}
\frac{d}{dt}\langle u, M u \rangle \le 2\alpha\langle u, A u \rangle + C\|\varphi\|_{L^2}^4,
\end{align}
where $C$ is a positive constant depending only on $n, \alpha$ but not on $m$.
\end{lemma}

Theorem \ref{blow} follows immediately from Lemmas \ref{virial-lemma1} and \ref{virial-lemma2}.

\begin{proof}[Proof of Lemma \ref{virial-lemma2}]
Using the identity $D_m^\alpha x = xD_m^\alpha  - \alpha D_m^{\alpha-2}  \nabla$, we first have the estimate:
\begin{align*}
[\,D_m^\alpha, M] = D_m^\alpha xD_m^{2-\alpha} x - xD_m^{2-\alpha}x D_m^\alpha = -\alpha(x\cdot \nabla + \nabla \cdot x).
\end{align*}

For a smooth function $v$ we get
\begin{align*}
[v, M] &= v x D_m^{2-\alpha} x - xD_m^{2-\alpha}x v\\
 &= v|x|^2D_m^{2-\alpha} - (2-\alpha)vx\cdot \nabla D_m^{-\alpha} - D_m^{2-\alpha}|x|^2v + (2-\alpha)D_m^{-\alpha}\nabla \cdot x v\\
 &= [\,|x|^2v, D_m^{2-\alpha}]  + (\alpha-2)\left(vx\cdot \frac{\nabla}{|\nabla|}|\nabla|D_m^{-\alpha} + |\nabla|D_m^{-\alpha}\frac{\nabla}{|\nabla|} \cdot x v\right).
\end{align*}
By density we may replace $v$ with $K_\alpha(|u|^2)$. We will show in the next section
\begin{align}\label{commut3}
|\langle u,[\,|x|^2K_\alpha(|u|^2), D_m^{2-\alpha}]u \rangle| \lesssim \|\varphi\|_{L^2}^4.
\end{align}
By the convolution estimate, Lemma \ref{weighted con} in case when $p = \infty$, $d_1 =  d_2 = \gamma$ and $f$ is radial, one have
\begin{align}\begin{aligned}\label{commut4}
&\qquad |\langle u, \left(vx\cdot \frac{\nabla}{|\nabla|}|\nabla| D_m^{-\alpha} + |\nabla|D_m^{-\alpha}\frac{\nabla}{|\nabla|} \cdot x v\right)u\rangle| \\
&\lesssim \|\psi\|_{L^\infty}\|\varphi\|_{L^2}^2\int |u(x)||x|^{-(\alpha -1)}\int|x-y|^{-(n-(\alpha -1))}|(\frac{\nabla}{|\nabla|}u)(y)|dydx
\end{aligned}\end{align}
To estimate this, we make use of the Stein-Weiss
inequality \cite{s-w}: for $f \in L^p$ with $1 < p < \infty$, $0 <
\lambda < n$, $\beta < \frac np$, and $n = \lambda + \beta$
\begin{align}\label{s-w ineq}
\||x|^{-\beta} (|\cdot|^{-\lambda}*f)\|_{L^p} \lesssim \|f\|_{L^p}.
\end{align}
Applying \eqref{s-w ineq} with $p = 2$, $\beta = \alpha - 1$ and $\lambda = n-(\alpha-1)$, \eqref{commut4} is bounded by $C\|\varphi\|_{L^2}^4$.

\end{proof}

\subsection{$L^2$ boundedness of commutator}
We show the commutator estimate \eqref{commut3}. We set $f = |x|^2K_\alpha(|u|^2)
$. From a simple calculation we observe that
\begin{align}\label{comm-m}
[D_m^{2-\alpha}, f]u(x) = m^{n+2-\alpha}\left([D_1^{2-\alpha},
f\left(\frac{\cdot}{m}\right)]u_m\right)(m x),
\end{align}
where $u_m(x) = m^{-n}u(\frac{x}{m})$. Thus we have the identity of
the operator norms
$$
\|[D_m^{2-\alpha}, f]\|_{L^2 \to L^2} =
m^{2-\alpha}\|[D_1^{2-\alpha}, f(\cdot/m)]\|_{L^2 \to L^2}.
$$

Set $f(x/m) = g(x)$. We define $T_i$, a pseudodifferential operator
of order $1-\alpha$,  by $ T_i = -
D_1^{2-\alpha}(-\Delta)^{-1}\partial_i $ so that $D_1^{2-\alpha} =
-\sum_{i = 1}^n T_i \partial_i$. Denote the kernel of $T_i
\partial_i$ by $k_i$. Then $[T_i\partial_i, g] = [T_i, g]\partial_i
+ T_i(\partial_i g)$ and the kernel $[T_i, g]$ is given by
$$K_i(x, y) = k_i(x, y)(g(y) - g(x)).$$
Suppose that $g$ is in Lipschitz class $\dot{\Lambda}^{2-\alpha}$. Then
$K_i$ is easily shown to be a Calder\'{o}n-Zygmund kernel. Here
$\|g\|_{\dot{\Lambda}^{2-\alpha}} = \sup_{x,y}\frac{|g(x) - g(y)|}{|x-y|^{2-\alpha}}$.
 We show that $[T_i, g]\partial_i$ is bounded in $L^2$ and its norm is
 bounded by a constant multiple of
$\|g\|_{\dot{\Lambda}^{2-\alpha}}$. By Theorem 3 in p. 294 of \cite{st}
and the duality of $[T_i, g]\partial_i$ we have only to show that
\begin{align}\label{cz-est}
\|[T_i, g]\partial_i (\zeta(\cdot/N))\|_{L^2} \lesssim
\|g\|_{\dot{\Lambda}^{2-\alpha}}N^\frac n2
\end{align}
for a fixed bump function $\zeta$ supported in the unit ball.
From the kernel estimate $|k_i(x, y)| \lesssim |x-y|^{-n + \alpha-1}$ it follows that $|K_i(x, y)|
\lesssim \|g\|_{\dot{\Lambda}^{2-\alpha}}|x-y|^{-(n-1)}$. If $|x| < 2N$,
then $$ |[T_i, g]\partial_i (\zeta(\cdot/N))(x)| \lesssim
\|g\|_{\dot{\Lambda}^{2-\alpha}}. $$ Thus $\|[T_i, g]\partial_i
(\zeta(\cdot/N))\|_{L^2(\{|x| < 2N\}} \lesssim
\|g\|_{\dot{\Lambda}^{2-\alpha}}N^\frac n2$. If $|x| \ge 2N$, then
$$|[T_i, g]\partial_i (\zeta(\cdot/N))(x)| \lesssim
\|g\|_{\dot{\Lambda}^{2-\alpha}}N^{n-1}|x|^{-(n-1)}.$$ Therefore
\begin{align*}
\|[T_i, g]\partial_i (\zeta(\cdot/N))\|_{L^2(\{|x| \ge 2N\}} &\lesssim \|g\|_{\Lambda^{2-\alpha}}N^{n-1}(\int_{|x| > 2N} |x|^{-2(n-1)}\,dx)^\frac12 \\
&\lesssim \|g\|_{\dot{\Lambda}^{2-\alpha}}N^\frac n2.
\end{align*}
This shows \eqref{cz-est} and thus $\|[T_i, g]\partial_i \|_{L^2\to
L^2} \lesssim \|g\|_{\dot{\Lambda}^{2-\alpha}} = m^{-(2-\alpha)}\|f\|_{\dot \Lambda^{2-\alpha}}$.
If $x \neq y$, then
$$
|f(x) - f(y)| \le |x-y|\int_0^1 |\nabla f (z_s)|\,ds, \quad z_s = x + s(y-x).
$$
Since $|\psi'(\rho)| \le C\rho^{-1}$ for $\rho > 0$, from Lemma \ref{weighted con} and mass conservation it follows that
$$
 |\nabla f (z_s)| \lesssim |z_s|^{1-\alpha}\|u\|_{L^2}^2 = ||x|-s|x-y||^{1-\alpha}\|\varphi\|_{L^2}^2,
$$
provided $\alpha < n-2$. By a simple calculation we see that if $0 < \theta < 1$, then
$$
\sup_{a > 0}\int_0^1|a - s|^{-\theta}\,ds \le C_\theta.
$$
Thus from this we get that
$$
|f(x) - f(y)| \lesssim |x-y|^{2-\alpha}\|\varphi\|_{L^2}^2,
$$
which implies that
\begin{align}\label{comm-lip}
\|[T_j, g]\partial_j\|_{L^2\to L^2} \lesssim m^{-(2-\alpha)}\|\varphi\|_{L^2}.
\end{align}

On the other hand, $T_i (\partial_i g)(u)(x) = \int k_i(x, y)
\partial_i g(y)u(y)\,dy$ and
$$
|T_i (\partial_i g)(u)(x)| \lesssim \int
|x-y|^{-(n-(\alpha-1))}|\partial_ig(y)||u(y)|\,dy.
$$

From the duality and Lemma \ref{weighted con}
\begin{align*}
|\langle u, T_j ((\partial_jg) u) \rangle| &= |\langle T_j^* u, (\partial_j g) u\rangle|\\
&\lesssim m^{-1}\|u\|_{L^2}\||(\partial_j) f(\cdot/m)|\int |\cdot-y|^{-(n-\alpha+1)}|u(y)|\,dy\|_{L^2}\\
&\lesssim m^{-(2-\alpha)}\|u\|_{L^2}^3\||\cdot|^{1-\alpha}|\int |\cdot-y|^{-(n-\alpha+1)}|u(y)|\,dy\|_{L^2},
\end{align*}
where $T_j^*$ is the dual operator of $T_j$. Using the Stein-Weiss inequality \eqref{s-w ineq} for $\beta = \alpha-1$, $\lambda = n-\alpha+1$ and $ p =2$, we get $|\langle u, T_j, \partial_j gu \rangle| \lesssim m^{-(2-\alpha)} \|u\|_{L^2}^4$. Thus
\begin{align}\label{comm-error}
\|T_j, \partial_j g\|_{L^2 \to L^2} \lesssim m^{-(2-\alpha)} \|\varphi\|_{L^2}^2.
\end{align}
Therefore from \eqref{comm-lip} and \eqref{comm-error} it follows that
$$
\|[D_m^{2-\alpha}, f]\|_{L^2 \to L^2} =
m^{2-\alpha}\|[D_1^{2-\alpha}, g]\|_{L^2 \to L^2} \le C\|\varphi\|_{L^2}^2.
$$
Here it should be noted that the constant $C$ does not depend on $m$.

\subsection{Propagation of the moment}
We finally show a propagation estimate of the moment. In what follows, Bessel potential estimates are used crucially. So, we introduce some basics of Bessel potential.

Let us denote the kernels of
Bessel potential $D^{-\beta}\;(\beta > 0)$ and $|\nabla|^\alpha D^{-\alpha} D^{-\beta}$ by
$G_\beta(x)$ and $K(x)$, respectively, where $D = \sqrt{1 - \Delta}$. Then
$$
K(x) = \sum_{k=0}^\infty A_k G_{2k+\beta}(x),
$$
where the coefficients $A_k$ is given by the expansion
$(1-t)^\frac\alpha2 = \sum_{k = 0}^\infty A_k t^k$ for $|t| < 1$
with $\sum_{k \ge 0}|A_k| < \infty$. One can show that $(1+
|x|)^\ell K \in L^1$ for $\ell \ge 1$. In fact, we have that for $2k+\beta < n$
\begin{align}\label{bessel1}
G_{2k+\beta}(x) \le C( |x|^{-n + \beta}\chi_{\{|x| \le 1\}}(x) +
e^{-c|x|}\chi_{\{|x| > 1\}}(x)).
\end{align}
And also from the integral representation of $G_{2k+\beta}$ such that
$$
G_{2k+\beta}(x) = \frac1{(4\pi)^{n/2}\Gamma(k+\beta/2)}\int_0^\infty
\lambda^{(2k+\beta-n)/2-1} e^{-|x|^2/4\lambda} e^{ -
\lambda}\,d\lambda
$$
we deduce that if $2k+\beta \ge n$, then
\begin{align}\label{bessel2}
G_{2k+\beta}(x) \le C(\chi_{\{|x| \le 1\}}(x) + e^{-c|x|}\chi_{\{|x|
> 1\}}(x)).
\end{align}
Here the constants $C$ of \eqref{bessel1} and \eqref{bessel2} are
independent of $k$. The functions $(1+|x|)^\ell G_{2k+\beta}$ have a uniform
integrable majorant on $k$ for each $\ell \ge 1$ and so $K$ does. For more details see p.132--135 of \cite{st-sing}.

We introduce the moment estimate
\begin{proposition}\label{mom prop}
Let $m > 0$ and $T^*$ be the maximal existence time of solution $u
\in C([0,T^*);H^k)$, $k = \max(\frac\gamma2, 4)$ to \eqref{main
eqn}. If $x\varphi, |x|\nabla\varphi \in L^2$, then $x u(t), |x|\nabla u(t) \in L^2$ for all
$t \in [0, T^*)$. Moreover, we have for $t \in [0, T^*)$
\begin{align*}
&\||x|u\|_{L^2} \le \||x|\varphi\|_{L^2} + Cm^{\alpha-3}\int_0^t\|u(t')\|_{H^2}\,dt',\\
&\||x|\nabla u\|_{L^2} \le \||x|\nabla \varphi\|_{L^2} +
Cm^{\alpha-3}\int_0^t \|u(t')\|_{H^3}\,dt',
\end{align*}
where $C$ does not depend on $m$.
\end{proposition}
For the proof for $\alpha = 1, 2$ see \cite{caz} for NLS and \cite{frohlenz2} for semirelativistic case.

\newcommand{\mm}{\mathbf{m}}
\begin{proof}[Proof of Proposition \ref{mom prop}]
We first consider the case $m > 0$. Let us denote $$\mm_\varepsilon(t) = \langle u(t), |x|^2 e^{-2\varepsilon |x|} u(t) \rangle$$ for $0 < \varepsilon \le m$.
From the regularity of $u$ and \eqref{comm-m} it follows that
\begin{align}\begin{aligned}\label{m-ep}
\mm_\varepsilon'(t) &= i m^{n+\alpha-2} \langle u_m, [D^\alpha, |x|^2 e^{-2\varepsilon |x|/m}]u_m\rangle\\
&= -2m^{n+\alpha-2}\,{\rm Im}\, \langle xe^{-\varepsilon|x|/m} u_m,
[D^\alpha, x e^{-\varepsilon|x|/m}]u_m\rangle,
\end{aligned}\end{align}
where $D = D_1 = \sqrt{1 - \Delta}$ and  $u_m(x) = m^{-n}u(x/m)$. Then
\begin{align*}
\langle xe^{-\varepsilon|x|/m} u_m, &[D^\alpha, x e^{-\varepsilon |x|/m}]u_m\rangle\\
&= \langle x e^{-\varepsilon |x|/m}u_m, [D^{\alpha-2}, x e^{-\varepsilon |x|/m}] D^2 u_m\rangle\\
&\qquad\quad + \langle D^{\alpha-2}(x e^{-\varepsilon |x|/m} u_m), [D^2, x e^{-\varepsilon |x|/m}] u_m\rangle\\
&\equiv I + I\!\!I.
\end{align*}
To handle $I$ set $\beta = 2-\alpha$ and denote the kernel of Bessel potential $D^{-\beta}$ by $G_{\beta}$. Then by mean value inequality such that
$|ye^{-\varepsilon|y|/m} - xe^{-\varepsilon|x|/m}| \lesssim |x-y|$,
we have
\begin{align*}
\quad &|([D^{-\beta}, xe^{-\varepsilon|x|/m}]D^2u_m)(x)|\\
&= \left|\int G_{\beta}(x-y) ye^{-\varepsilon|y|/m}D^2u_m(y)dy - xe^{-\varepsilon|x|/m}\int G_{\beta}(x-y)D^2u_m(y)dy\right|\\
&= \left|\int G_{\beta}(x-y) (ye^{-\varepsilon|y|/m} - xe^{-\varepsilon|x|/m})D^2u_m(y)dy\right|\\
&\lesssim \int G_{\beta}(x-y)|x-y||D^2u_m(y)|dy.
\end{align*}
Since $|x|G_{\beta}$ is integrable, from Cauchy-Schwarz inequality and Young's inequality it follows that
\begin{align}\label{I}
|I| \le C m^{-\frac n2 - 2}\|u\|_{H^2}\sqrt{\mathbf
m_\varepsilon},
\end{align}
where $C$ is independent of $\varepsilon$ and $m$.

Now using Cauchy-Schwarz inequality we estimate $I\!\!I$ as
follows:
\begin{align}\begin{aligned}\label{II}
|I\!\!I| &= |\langle D^{-\beta}( x e^{-\varepsilon |x|/m} u_m), [D^2,
xe^{-\varepsilon |x|/m}]  u_m\rangle|\\
&= |\langle D^{-\beta}( x e^{-\varepsilon |x|/m} u_m), (\Delta(x e^{-\varepsilon |x|/m}) + 2\nabla(x e^{-\varepsilon |x|/m})\cdot\nabla)u_m\rangle|\\
& \le C m^{-n}\|u\|_{H^1}\sqrt{\mathbf
m_\varepsilon},
\end{aligned}\end{align}
where $C$ is independent of $\varepsilon$ and $m$. We have used the fact
$$|\Delta(x e^{-\varepsilon |x|/m}) + 2\nabla(x e^{-\varepsilon |x|/m}| \le C.$$

Substituting the estimates for $I$ and $I\!\!I$ into \eqref{m-ep},
we have
$$
\mathbf m_\varepsilon \le \mathbf m_\varepsilon(0) +
Cm^{\alpha-3}\int_0^t \|u(t')\|_{H^2} \sqrt{\mathbf
m_\varepsilon(t')}dt'.
$$
Gronwall's inequality yields
$$
\sqrt{\mathbf m_\varepsilon} \le \sqrt{\mathbf m_\varepsilon(0)} +
Cm^{\alpha-3}/2\int_0^t \|u(t')\|_{H^2}dt'.
$$
Thus letting $\varepsilon \to 0$, it
follows that
\begin{align}\label{m-est1}
\||x|u\|_{L^2} \le \||x|\varphi\|_{L^2} + Cm^{\alpha-3}/2\int_0^t
\|u(t')\|_{H^2}\,dt'\;\;\mbox{for all}\;\;t \in [0, T^*).
\end{align}

Now let us observe that $u \in H^3$, set $v = \partial_j u$. Then one
can easily show that
$$
\frac d{dt}\langle v \, |x|^2e^{-2\varepsilon|x|} v\rangle =
i\langle v, [D_m^\alpha, |x|^2e^{-2\varepsilon|x|}] v\rangle.
$$
So, by the same estimates as above we get
\begin{align}\label{m-est2}
\||x|\nabla u\|_{L^2} \lesssim \||x|\nabla \varphi\|_{L^2} +
m^{\alpha-3}\int_0^t \|u(t')\|_{H^3}\,dt'\;\;\mbox{for all}\;\;t
\in [0, T^*).
\end{align}

\end{proof}

\bigskip

\section*{Acknowledgments} Y. Cho and G. Hwang were supported by National Research Foundation
of Korea Grant funded by the Korean Government (2011-0005122).

\end{document}